\documentclass{amsart}

\usepackage{hyperref}   
\usepackage[square,numbers]{natbib}     
\usepackage{amsthm}     
\usepackage{amsmath}    
\usepackage{amssymb}    
\usepackage{amsxtra}    
\usepackage{eucal}      
\usepackage[bbgreekl]{mathbbol}   
\usepackage{todonotes}
\usepackage{comment}
\usepackage{tikz}

\usetikzlibrary{matrix}
\usetikzlibrary{patterns, shapes}
\usetikzlibrary{arrows}
\usetikzlibrary{calc,3d}
\usetikzlibrary{decorations,decorations.pathmorphing, decorations.pathreplacing}
\usetikzlibrary{through}
\tikzset{ext/.style={circle, draw,inner sep=1pt},int/.style={circle,draw,fill,inner sep=1pt},nil/.style={inner sep=1pt}}
\tikzset{exte/.style={circle, draw,inner sep=3pt},inte/.style={circle,draw,fill,inner sep=3pt}}
\tikzset{diagram/.style={matrix of math nodes, row sep=3em, column sep=2.5em, text height=1.5ex, text depth=0.25ex}}
\tikzset{diagram2/.style={matrix of math nodes, row sep=0.5em, column sep=0.5em, text height=1.5ex, text depth=0.25ex}}

\usepackage{tikz-cd}

\theoremstyle{plain}

\newtheorem{introthm}{Theorem}

\newtheorem{thm}[subsection]{Theorem}
\newtheorem{prop}[subsection]{Proposition}
\newtheorem{lemm}[subsection]{Lemma}
\newtheorem{cor}[subsection]{Corollary}

\title[Rational homotopy theory of operad modules]{Rational homotopy theory of operad modules through colored operads}

\date{\today}

\author{Thomas Willwacher}
\address{Department of Mathematics\\
ETH Z\"urich\\
R\"amistrasse 101\\
8092 Z\"urich, Switzerland}
\email{thomas.willwacher@math.ethz.ch}

\DeclareMathOperator{\QQ}{\mathbb{Q}}

\DeclareMathOperator{\Mod}{\mathcal{M}\mathit{od}}  
\DeclareMathOperator{\BiMod}{\mathcal{B}\mathit{i}\!\Mod}  
\newcommand{\dgHBiModc}{\dg\!H\!\BiMod^c}  
\newcommand{\dgBiModc}{\dg\!\BiMod^c}

\DeclareMathOperator{\Seq}{\mathcal{S}\mathit{eq}}  
\DeclareMathOperator{\Op}{\mathcal{O}\mathit{p}}    

\DeclareMathOperator{\f}{\mathit{f}}        
\DeclareMathOperator{\dg}{\mathit{dg}}      
\DeclareMathOperator{\dgca}{\mathit{dgca}}

\DeclareMathOperator{\Map}{\mathtt{Map}}
\DeclareMathOperator{\Mor}{\mathtt{Mor}}

\DeclareMathOperator{\Hom}{\mathtt{Hom}}

\DeclareMathOperator{\eq}{\mathrm{eq}}

\DeclareMathOperator{\FreeOp}{\mathbb{F}}     

\DeclareMathOperator{\DGOmega}{\mathtt{\Omega}}

\DeclareMathOperator{\DGB}{\mathtt{B}}

\DeclareMathOperator{\DGG}{\mathtt{G}}

\DeclareMathOperator{\DGL}{\mathtt{L}}

\newcommand{\CdgSeq}{\colC\!\dg\!\Seq}

\DeclareMathAlphabet{\mathsfit}{OT1}{cmss}{m}{sl}
\DeclareMathOperator{\AOp}{\mathsfit{A}}
\DeclareMathOperator{\BOp}{\mathsfit{B}}
\DeclareMathOperator{\COp}{\mathsfit{C}}
\DeclareMathOperator{\DOp}{\mathsfit{D}}

\DeclareMathOperator{\MOp}{\mathsfit{M}}
\DeclareMathOperator{\NOp}{\mathsfit{N}}
\DeclareMathOperator{\POp}{\mathsfit{P}}
\DeclareMathOperator{\QOp}{\mathsfit{Q}}

\DeclareMathOperator{\SOp}{\mathsfit{S}}

\newcommand{\bbS}{\mathbb S}
\renewcommand{\Bar}{\DGB}

\newcommand{\dgOpc}{\dg\!\Op^c}
\newcommand{\CdgOpc}{\colC\!\dgOpc}

\newcommand{\dgHOpc}{\dg\!\mathit{H}\!\Op^c}
\newcommand{\CdgHOpc}{\colC\!\dgHOpc}

\newcommand{\G}{\DGG}

\newcommand{\La}{\Lambda}
\newcommand{\Lai}{\Lambda\!}

\newcommand{\sSetOp}{\sSet\!\Op}
\newcommand{\CsSetOp}{\colC\sSetOp}
\newcommand{\Trip}{\mathcal{T}\mathit{rip}}
\newcommand{\sSetTrip}{\sSet\Trip}

\newcommand{\dgHTripc}{\dg\!H\Trip^c}

\newcommand{\TOp}{\mathsf T}
\newcommand{\sSetSeq}{\sSet\!\Seq}
\newcommand{\CsSetSeq}{\colC\sSetSeq}
\newcommand{\sSet}{\mathit{s}\mathcal{S}\mathit{et}}
\newcommand{\dgCom}{\mathit{dg}\mathcal{C}\mathit{om}}
\DeclareMathOperator{\Ind}{\mathtt{Ind}}  

\newcommand{\dgVect}{\dg\!\mathcal{V}\mathit{ect}}

\newcommand{\catC}{{\mathcal C}}
\newcommand{\catD}{{\mathcal D}}
\newcommand{\catA}{{\mathcal A}}
\newcommand{\catB}{{\mathcal B}}

\renewcommand{\leftrightarrows}{\rightleftarrows}
\newcommand{\colC}{{\mathbf C}}

\newcommand{\IBiMod}{\mathcal I\!\BiMod}
\newcommand{\AssM}{AssM}
\newcommand{\AssMc}{AssM^c}
\newcommand{\dgHIBiModc}{\dg\!H\IBiMod^c}

\begin{document}

\begin{abstract}
	We extend the rational homotopy theory of operads developed by B. Fresse to several types of modules over operads.
\end{abstract}

\maketitle

\section{Introduction}
The rational homotopy theory of operads has been developed by B. Fresse \cite{OperadHomotopyBook,ExtendedRHT}.
Concretely, he constructed model category structures on the category $\sSetOp$ of simplicial operads and the category $\dgHOpc$ of dg Hopf cooperads, i.e., cooperads in dg commutative algebras. These categories then fit into a Quillen adjunction
\begin{equation}\label{equ:RHT adjunction}
\G : \dgHOpc \leftrightarrows \sSetOp^{op}  : \DGOmega_\sharp
\end{equation}
extending the standard Quillen adjunction between simplicial sets and dg commutative algebras in rational homotopy theory.

The main goal of this note is to extend the above construction to operadic (bi)modules. In particular, we construct model category structures on various categories of dg Hopf cooperadic comodules, and show that they fit into Quillen adjunctions similar to \eqref{equ:RHT adjunction} above.

We do this by first noting that the adjunction \eqref{equ:RHT adjunction} naturally extends to colored operads, with a finite set of colors $\colC$, to yield a Quillen adjunction 
\begin{equation}\label{equ:RHT adjunction col}
    \G : \CdgHOpc \leftrightarrows \CsSetOp^{op}  : \DGOmega_\sharp.
\end{equation}

Let $\POp$, $\QOp$ be operads, and let $\MOp$ be an operadic $\POp$-$\QOp$-bimodule. 
Then the triple $(\POp, \MOp, \QOp)$ naturally generates a two-colored operad with $\POp$ the operations of input and output color I, $\QOp$ the operations of input and output color II, and $\MOp$ the operations of input color II and output color I.
By suitable restriction of the categories in \eqref{equ:RHT adjunction col} we are hence able to obtain a similar operadic adjunction for the category of operadic bimodules. We summarize our main results in the following theorem.

\begin{introthm}\label{thm:main bi}
\begin{itemize}
\item (Berger-Moerdijk \cite{BMColored}) For $\POp$, $\QOp$ simplicial operads there is a cofibrantly generated model structure on the category $\BiMod_{\POp, \QOp}$ of operadic $\POp$-$\QOp$-bimodules such that the weak equivalences (resp. fibrations) are the arity-wise weak equivalences (resp. fibrations) of simplicial sets.
\item For $\COp$, $\DOp$ dg Hopf cooperads there is a cofibrantly generated model structure on the category $\dgHBiModc_{\COp, \DOp}$ of dg Hopf cooperadic $\COp$-$\DOp$-bicomodules such that the weak equivalences are the arity-wise quasi-isomorphisms.
\item There is a Quillen adjunction 
\[
	\G_\bullet : \dgHBiModc_{\DGOmega_\sharp(\POp), \DGOmega_\sharp(\QOp)} \leftrightarrows \BiMod_{\POp, \QOp}^{op}  : \DGOmega_\sharp.
\]
with $\G_\bullet$ the arity-wise application of Quillen's realization functor, see \eqref{equ:G def} below.
\item Let $\MOp$ be a cofibrant $\POp$-$\QOp$-bimodule.
If $\POp(1)$ and $\QOp(1)$ are connected and all simplicial sets $\POp(r)$, $\QOp(r)$ and $\MOp(r)$ have finite dimensional rational cohomology in every degree, then there is a natural comparison weak equivalence 
\[
	\DGOmega_\sharp(\MOp)(r) \xrightarrow{\sim}
	\DGOmega(\MOp(r))
\]
to the standard (Sullivan) differential forms functor $\DGOmega(-)$ of rational homotopy theory.
\end{itemize}
\end{introthm}

We also show an extension valid for unital operads, that is, operads that satisfy $\POp(0)=\QOp(0)=*$, see section \ref{sec:La operads} below.
By restricting to the case $\POp=1$ (resp. $\QOp=1$) we also obtain a similar result for operadic right (resp. left) modules. 
We also show a variant of Theorem \ref{thm:main bi} for infinitesimal bimodules, and generalizations thereof, see section \ref{sec:mixedmod} below.
Our methods are flexible and generic, and can be applied to other types of operadic modules, that can be encoded into colored operads.

Mind however, that we are generally restricted to situations in which the spaces of nullary operations are either empty or consist of one point. In the applications we mainly target, the operads and modules arise as configuration spaces of points on manifolds, and this restriction is automatically satisfied.
However, we do not cover, for example, the classical situation of algebras over an operad, which can be considered as operadic left modules concentrated in arity zero.

\subsection*{Sketch of the construction}
Note that a $\POp$-$\QOp$-bimodule $\MOp$ gives rise to an object in each of the following four categories.
\begin{itemize}
\item The category of two-colored operads, with the bimodule encoded as a colored operad as described above.
\item The category of triples $(\POp, \MOp, \QOp)$ consisting of two operads and a bimodule.
\item If we want fo fix the operads $\POp$ and $\QOp$, then we may consider the under-category whose objects are arrows of triples $(\POp, \emptyset, \QOp)\to (\POp', \MOp, \QOp')$, and consider our bimodule as an object in this category.
\item Finally, we may remove the operads from the data and just consider the bimodule as an object in the category of $\POp$-$\QOp$-bimodules.
\end{itemize}

In this paper we start from the category of colored operads, for which model structures and a rational homotopy have been (essentially) constructed by B. Fresse. Using natural categorial constructions we then prolong this theory to each of the other categories above. This serves a dual purpose: First, it yields a simplified construction of the rational homotopy theory of operadic bimodules, in that we do not have to repeat some amount of elaborate technical verifications by Fresse, and instead can just cite some general (model) categorial results. 
The second purpose is that this line of attack also allows us to compare the above four ways of encoding an operadic bimodule, by studying the natural inclusion functors between these categories.

\subsection*{Structure of the paper}
In section \ref{sec:preliminaries} we state our conventions and recall some preliminaries and standard model categorial constructions we use in later sections.

Section \ref{sec:operad rht} contains a brief recollection of the rational homotopy theory for operads developed by the first author. We note in passing that it can be extended (virtually) without changes to colored operads on a finite set of colors.

This is then specialized in section \ref{sec:triples} to obtain a rational homotopy theory of triples $(\POp, \MOp,\QOp)$ consisting of two operads and a bimodule.
Finally section \ref{sec:bimodules} contains the further restriction to bimodules (fixing the operads), and a proof of Theorem \ref{thm:main bi}.
Section \ref{sec:mixedmod} contains extensions to other types of modules, including infinitesimal bimodules.

\subsection*{Acknowledgements}
The author is greatly indebted to Benoit Fresse, for laying the foundations for this work, and for many helpful and clarifying discussions.

The work was furthermore supported by the NCCR Swissmap, funded by the Swiss National Science Foundation.

\section{Preliminaries}\label{sec:preliminaries}

\subsection{Notation }\label{sec:notation}
In the following we generally use the conventions and notation of B. Fresse's book ``Homotopy of operads and Grothendieck--Teichm\"uller groups'' \cite{OperadHomotopyBook}, with slight adjustments.

We denote by $\sSet$ the category of simplicial sets. We equip the category $\sSet$ with a model structure such that the weak equivalences are weak homotopy equivalences, the fibrations are Kan fibrations, and the cofibrations are degree-wise injective maps, cf. \cite[Theorem 1.3.12]{OperadHomotopyBook}.

The category $\dgVect$ of non-negatively graded cochain complexes is equipped with a model structure such that the weak equivalences are quasi-isomorphisms, the fibrations are the degreewise surjective maps, and the cofibrations are the maps injective in all positive degrees. We emphasize that in contrast to the terminology of \cite{OperadHomotopyBook} we do not call dg vector spaces dg modules, but rather reserve the word module to refer to operadic modules only, in order to avoid confusion.

A symmetric sequence $\MOp$ in a category $\catC$ is a collection $\MOp(r)$ of objects in $\catC$, equipped with right actions of the symmetric group $\bbS_r$, for each $r= 1,2,\dots$. In particular, note that our symmetric sequences usually have no terms in arity $r=0$. (We will briefly deviate from this convention in section \ref{sec:La operads}, though.)
Assuming that $\catC$ is symmetric monoidal and has finite limits, the category $\catC\Seq$ of symmetric sequences in $\catC$ is equipped with the "plethysm" monoidal product $\circ$.
Concretely for $\MOp,\NOp\in \catC\Seq$ we have
\[
(\MOp\circ\NOp)(r) 
=
\coprod_{k}
\MOp(k) \otimes_{\bbS_r}
\left(\coprod_{r_1+\cdots+r_k=r}
\Ind_{\bbS_{r_1} \times \cdots \times \bbS_{r_k}}^{\bbS_r}
(\NOp(r_1) \otimes \cdots \otimes \NOp(r_k))
\right).
\]
An operad $\POp$ in $\catC$ can then be defined as a monoid in $(\catC\Seq,\circ)$.
An operadic left (resp. right) $\POp$ module is then a left or right module for the monoid $\POp$ in $\catC\Seq$.
Similarly, for $\QOp$ another operad, an operadic $\POp$ -$\QOp$-bimodule is a bimodule for the monoids $\POp$ and $\QOp$.

For later use we shall also define a restricted plethysm product $\circ'$ on the category $\catC\Seq$, such that
\[
(\MOp\circ'\NOp)(r) 
=
\coprod_{k+l=r+1}
\Ind_{\bbS_{k-1} \times \bbS_{l}}^{\bbS_r}\left(
\MOp(k) \otimes \NOp(l)
\right).
\]
This is the infinitesimal version of $\circ$ in the sense that formally all but one of the $\NOp$ have been replaced by the monoidal unit object.

An infinitesimal $\POp$-bimodule is a right $\POp$-module $\MOp$, together with compatible partial left actions, i.e., a morphism of the form
\[
\POp \circ' (\MOp \circ\POp) \to \MOp,
\] 
satisfying natural compatibility relations.
We refer to \cite[section 3]{AroneTurchin} for more details.

\subsection{Colored symmetric sequences and operads}
We shall heavily use colored operads, on a finite set of colors $\colC$.
A colored symmetric sequence $\MOp$ in a category $\catD$ is a collection of objects 
\[
\MOp(\underline r; c) \in \catD
\]
with an action of the symmetric group $\bbS_{\underline r}=\prod_{d\in C} \bbS_{r_d}$, where $c\in \colC$ and 
\begin{gather*}
    \underline r: \colC \to \mathbb Z_{\geq 0} \\
    d \mapsto r_d
\end{gather*}
is a function. 
We consider $\MOp(\underline r; c)$ as a space of operations with $r_d$ inputs of color $d$ ($d\in \colC$) and output of color $c$.
We usually consider only colored symmetric sequences with no nullary operations, i.e., we require that $\sum_dr_d>0$.
In case there is a natural ordering on the set of colors $\colC=\{c_1,\dots,c_n\}$ we consider $\underline r=(r_{c_1},\dots,r_{c_n})$ as a tuple and also use the notation 
\[
\MOp(r_{c_1},\dots,r_{c_n}; c).
\]
For the unary operations we also use the notation 
\[
\MOp(d; c)=\MOp(0,\dots,0,1,0,\dots,0; c).
\]
Assuming that $\catC$ is symmetric monoidal,
the category $\colC\catC\Seq$ of colored symmetric sequences has a natural monoidal structure $\circ$ such that 
\[
(\MOp\circ \NOp)(\underline r;c)
=
\bigoplus_{\underline s}
\MOp(\underline s;c)
\otimes_{\bbS_{\underline s}}
\left(
\bigoplus_{
	\underline t^{d,i} (d\in \colC, i=1,\dots,s_c) 
	\atop \sum \underline t^{d,i} = \underline r
}
\Ind^{\bbS_{\underline r}}_{\prod\bbS_{\underline t^{d,i}}}
\bigotimes_{d,i} \NOp(\underline t^{d,i};d)
\right).
\] 
A colored operad is a monoid in $\colC\catC\Seq$ with respect to this monoidal structure.

\subsection{Model categorial transfer}
We assume that all considered categories are complete and cocomplete without further mention.
Let $\catC$ be a cofibrantly generated model category, with generating cofibrations $I$ and generating acyclic cofibrations $J$.
Let 
\[
F\colon \catC \leftrightarrows \catD \colon G
\]
be an adjunction. 
Then the transfered model category structure on $\catD$, if it exists, is defined by the following distinguished classes of morphisms.
\begin{itemize}
\item A morphism $f$ in $\catD$ is a weak equivalence (resp. fibration) if $G(f)$ is a weak equivalence (resp. fibration).
\item The cofibrations in $\catD$ are those morphisms that satisfy the left-lifting property with respect to all fibrations.
\item The generating cofibrations (resp. the generating acyclic cofibrations) in $\catD$ are $F(I)$ (resp. $F(J)$).
\end{itemize}



\begin{thm}[{Kan, see \cite[Theorem 4.3.3]{OperadHomotopyBook}, \cite[Theorem 7.4.4]{HeutsMoerdijk}}]
	\label{thm:transfer Kan}
The above distinguished classes of morphisms define a cofibrantly generated model category structure on $\catD$ if the following conditions hold:
\begin{enumerate}
    \item The domains of the morphisms $F(i) \in F(I)$ (respectively, $F(j) \in F(J)$) are small with respect to the relative $F(I)$-cell complexes (respectively, $F(J)$-cell complexes) in the category $\catD$.
    This is true in particular if the right-adjoint functor $G$ preserves filtered colimits.
    \item The image of any relative $F(J)$-cell complex under the right adjoint functor $G$ forms a weak-equivalence in $\catC$. Explicitly, this means that:
    \begin{enumerate}
    \item Any pushout of the image $F(j)$ of any generating acyclic cofibration $j\in J$ is a weak equivalence in $\catD$.
    \item Any transfinite composition of morphisms as in (a) is a weak equivalence in $\catD$.
\end{enumerate}
\end{enumerate}

\end{thm}

We will use the following simple criterion to check condition (2) of the Theorem.

\begin{lemm}\label{lem:cell we criterion}
Consider a commutative diagram\footnote{This means that the diagrams of left and right adjoints each commute.} of adjunctions (with $L$, $L'$ and $\iota$ the left adjoints)
\[
\begin{tikzcd}[column sep=2cm, row sep=1.5cm]
	\catC 
	\ar[shift left]{r}{L} 
	\ar[shift left]{dr}{L'} 
	&
	 \ar[shift left]{l}{R} \catD 
	 \ar[shift left]{d}{\iota}
	 \\
	& \catA \ar[shift left]{u}{\pi}
	\ar[shift left]{ul}{R'} 
\end{tikzcd}.
\]
Suppose that:
\begin{itemize}
	\item $\COp$ is a cofibrantly generated model category with generating acyclic cofibrations $J$, $\AOp$ is a model category and $(L', R')$ is a Quillen adjunction.
	\item $\pi\iota= \mathit{id}$.
	\item $R'$ preserves weak equivalences.
\end{itemize}
Then condition (2) of Theorem \ref{thm:transfer Kan} holds for the adjunction $(L,R)$.
\end{lemm}
\begin{proof}
Let $f$ be a relative $L(J)$-cell complex in $\catD$. We have to check that $R(f)$ is a weak equivalence in $\catC$.
Note that since $\iota$ preserves colimits by adjunction we have that $\iota(f)$ is a relative $\iota(L(J))$-cell complex. But by commutativity of the diagram of left adjoints $\iota(L(J))=L'(J)$. But $L'$ is left Quillen and hence $L'(J)$ consists of acyclic cofibrations. Hence $\iota(f)$ is a weak equivalence. Hence $R'(\iota(f))$ is a weak equivalence by the last assumption. But by commutativity of the diagram of right adjoints
\[
	R'(\iota(f)) = R(\pi(\iota(f)))=R(f),
\]
so that $R(f)$ is a weak equivalence as was to be shown.
\end{proof}

\begin{lemm}\label{lem:cell we criterion 2}
	Let $\COp$ be a cofibrantly generated model category with generating acyclic cofibrations $J$ and $\AOp$ another model category, that fit into a diagram of functors.
	\[
	\begin{tikzcd}[row sep=1.5cm]
		\catC 
		\ar[shift left]{rr}{L} 
		\ar{dr}{\alpha} 
		& &
		 \ar[shift left]{ll}{R} \catD 
		 \ar{dl}{\beta}
		 \\
		& \catA &
	\end{tikzcd}.
	\]
	Suppose that:
	\begin{itemize}
		\item $(L,R)$ is an adjunction.
		\item We have $\alpha=\beta\circ L$.
		\item $\alpha$ preserves acyclic cofibrations and $\beta$ preserves colimits. 
		\item For $f$ a morphism in $\catD$ we have that $R(f)$ is a weak equivalence iff $\beta(f)$ is a weak equivalence. For example, this holds if $\beta=\alpha\circ R$ and $\beta$ creates weak equivalences.
	\end{itemize}
	Then condition (2) of Theorem \ref{thm:transfer Kan} holds for the adjunction $(L,R)$.
	\end{lemm}
	\begin{proof}
		Let $f$ be an $L(J)$-relative cell complex in $\catD$. We have to check that $R(f)$ is a weak equivalence.
		Since $\beta$ preserves colimits, $\beta(f)$ is a relative $\beta(L(J))$-cell complex, hence a relative $\alpha(J)$-cell complex by the second assumption. But since $\alpha$ preserves acyclic cofibrations, $\beta(f)$ is a weak equivalence.
		Hence by the last assumption $R(f)$ is a weak equivalence.
	\end{proof}

\subsection{Reflective subcategories}
\label{sec:restriction}
A reflective subcategory $\catD\subset \catC$ is a full subcategory such that the inclusion functor $\iota$ has a left adjoint.
\[
\pi \colon \catC \leftrightarrows \catD \colon \iota.	
\]
In this case the counit of the adjunction $\pi\circ\iota\Rightarrow\mathit{id}$ is an isomorphism.
Dually, a full subcategory is coreflective if the inclusion has a right adjoint.

\begin{lemm}\label{lem:fully faithful}
Let $\catD\subset \catC$ be a reflective (resp. coreflective) subcategory, with $\pi$ the left (resp. right) adjoint to the inclusion $\iota$.
Suppose that this adjunction is Quillen with respect to model category structures on $\catC$ and $\catD$, and that $\pi$ preserves weak equivalences.
Then the inclusion $\catD\subset \catC$ is homotopically fully faithful, that is,
\[
	\Map^h_{\catD}(A, B) \simeq \Map^h_{\catC}(\iota^h(A), \iota^h(B))
\]
for any pair of objects $A,B\in \catD$.
\end{lemm}
\begin{proof}
We only show the statement for $\catD\subset \catC$ a reflective subcategory. The coreflective case follows by duality.
For $A\in \catD$ fibrant and $X\xrightarrow{\sim}\iota(A)$ a cofibrant replacement of $\iota(A)$ in $\catC$ consider the derived adjunction counit 
\[
\pi(X) \to \pi\iota(A) \to A.
\]
The first arrow is a weak equivalence since $\pi$ preserves weak equivalences, and the second is an isomorphism by reflectivity. It follows that the derived adjunction counit is a weak equivalence.
Also note that $\pi(X)\to A$ is then a cofibrant replacement of $A$ in $\catD$.

Let $B\in \catD$ be fibrant (w.l.o.g.) and let $B^{\Delta}$ be a simplicial frame.
Then $\iota(B^\Delta)=:\iota(B)^\Delta$ is a simplicial frame of the fibrant object $\iota(B)$ in $\catC$ by Quillen adjunction. We hence have
\begin{align*}
\Map^h_{\catD}(A, B)
&\simeq
\Mor_{\catD}(\pi(X), B^\Delta)
=
\Mor_{\catC}(X, \iota(B^\Delta))
=
\Mor_{\catC}(X, \iota(B)^\Delta)
\\&\simeq 
\Map_{\catC}(X, \iota(B))
\simeq
\Map_{\catC}^h(\iota^h(A), \iota^h(B))
\end{align*}

\end{proof}


We also want to restrict (Quillen) adjunctions to (co)reflective subcategories. This can often be done by the following elementary lemma.

\begin{lemm}\label{lem:restr adj}
    Let 
    \[
    L \colon \catC \leftrightarrows \catD \colon R
    \]
    be an adjunction and let 
    \begin{align*}
        \iota_{\catA} \colon \catA &\leftrightarrows \catC \colon \pi_{\catA}
        &
        \iota_{\catB} \colon \catB &\leftrightarrows \catD \colon \pi_{\catB}
    \end{align*}
    be coreflective subcategories. Suppose that the left adjoint $L$ lifts to a functor $L':\catA\to\catB$ satisfying $L\circ \iota_{\catA} =\iota_{\catB}\circ L'$, so that we have a diagram 
    \[
    \begin{tikzcd}
    \catA\ar[shift left]{r}{L'} 
    \ar[shift left]{d}{\iota_{\catA}}
    & \catB 
    \ar[shift left]{d}{\iota_{\catB}}
    \\
    \catC \ar[shift left]{r}{L} 
    \ar[shift left]{u}{\pi_{\catA}}
    & 
    \catD
    \ar[shift left]{u}{\pi_{\catB}}
    \ar[shift left]{l}{R}
    \end{tikzcd}
    \]
    Then the functor $R':=\pi_{\catA}\circ L \circ \iota_{\catB}$ is right-adjoint to $L'$, and the resulting square of adjunctions commutes, in the sense that the diagrams of left and right adjoints both commute.
    
    If in addition all four categories are model categories such that the adjunctions $(L,R)$ and $(\iota_{\catA}, \pi_{\catA})$ are Quillen and $\iota_B$ creates (acyclic) cofibrations, then $(L',R')$ is a Quillen adjunction as well.
    \end{lemm}
\begin{proof}
We check the adjunction relation for objects $A\in \catA$ and $B\in \catB$.
\begin{align*}
\Mor_{\catB}(L'A, B) 
&= 
\Mor_{\catD}(\iota_{\catB}L'A, \iota_{\catB}B) 
=
\Mor_{\catD}(L\iota_{\catA}A, \iota_{\catB}B)
\\&=
\Mor_{\catC}(\iota_{\catA}A, R\iota_{\catB}B)
=
\Mor_{\catA}(A, \pi_{\catA} R\iota_{\catB} B)
\\
&=\Mor_{\catA}(A, R' B).
\end{align*}

Here we used fully faithfulness of $\iota_{\catB}$, and then the adjunction relations for $L$ and $\iota_{\catA}$.
The identification is clearly functorial in $A$ and $B$ and hence $(L',R')$ is an adjunction.

For the second assertion let $f$ be an (acyclic) cofibration in $\catA$. We need to  check that $L'f$ is again an (acyclic) cofibration. By assumption this means that $\iota_B L'f$ is an acylic cofibration. Since adjoints are unique and the diagram of left-adjoints commutes by assumptions, so must the diagram of right adjoints, $R'\pi_{\catB}=\pi_{\catA}R$.

We finally check that $(L',R')$ is Quillen under the stated hypothesis, by checking that for $f$ a(n acyclic) cofibration in $\catA$ we have that $L'f$ is an acyclic codibration in $\catB$.
By assumption this is the same as $\iota_{\catB}L'f=L\iota_{A}f$ being an acyclic cofibration in $\catD$.
But this is true since $\iota_{A}$ and $L$ preserve (acyclic) cofibrations. 

\end{proof}
We note that the existence of the right adjoint $R'$ in the above lemma could as well be deduced from the adjoint functor lifting theorem, using that coreflective subcategory inclusions are comonadic functors.

\subsection{Recollection of generalities on the slice model structure}

We recall the following well-known results on slice categories.

\begin{thm}[Hirschhorn \cite{HirschhornSlice}]
	\label{thm:slice model str}
Let $\catC$ be a cofibrantly generated model category with generating cofibrations $I$ and generating acyclic cofibrations $J$. Let $A\in\catC$ be an object. Then the following holds:
\begin{itemize}
	\item 
The undercategory $\catC^{A/}$ is a cofibrantly generated model category, with the following classes of distinguished morphisms:
\begin{itemize}
	\item The weak equivalences (resp. fibrations, cofibrations) are those morphisms that are weak equivalences (resp. fibrations, cofibrations) in $\catC$.
	\item The generating cofibrations (resp. acyclic cofibrations) are the coproducts $i\sqcup A$ of the generating cofibrations $i\in I$ (resp. acyclic cofibrations $i\in J$) of $\catC$.
\end{itemize}
\item The overcategory $\catC_{/A}$ is a cofibrantly generated model category, with the following classes of distinguished morphisms:
\begin{itemize}
	\item The weak equivalences (resp. fibrations, cofibrations) are those morphisms that are weak equivalences (resp. fibrations, cofibrations) in $\catC$.
	\item The generating cofibrations (resp. acyclic cofibrations) are the morphisms 
	\[
	\begin{tikzcd}	
	X \ar{rr}{i} \ar{dr} & & Y \ar{dl} \\
	& A & 
	\end{tikzcd}
	\]
	with $i\in I$ (resp. $i\in J$).
\end{itemize}
\end{itemize}
\end{thm}

\begin{prop}[Base change]\label{prop:slice base change}
Let $f:A\to B$ be a morphism in $\catC$. Then the following holds.
\begin{itemize}
\item There is a Quillen adjunction 
\[
f_* \colon 
\catC^{A/}
\leftrightarrows 
\catC^{B/}
\colon 
f^*
\]
with $f^*$ composition with $f$ and $f_*$ pushout along $f$.
This is a Quillen equivalence if $f$ is a weak equivalence and either of the following holds: 
\begin{itemize}
\item $\catC$ is left proper.
\item $f$ is an acyclic cofibration.
\item $A$ and $B$ are cofibrant objects.
\end{itemize}
\item There is a Quillen adjunction 
\[
f_* 
\colon 
\catC_{/A}
\leftrightarrows 
\catC_{/B}
\colon 
f^*
\]
with $f_*$ composition with $f$ and $f^*$ pullback along $f$.
This is a Quillen equivalence if $f$ is a weak equivalence and either of the following holds: 
\begin{itemize}
\item $\catC$ is right proper.
\item $f$ is an acyclic fibration.
\item $A$ and $B$ are fibrant objects.
\end{itemize}
\end{itemize}
\end{prop}
\begin{proof}
The adjunction property is \cite[Lemma 7.6.6]{HirschhornBook}. 
The remaining statements are essentially found in \cite[Proposition 3, Corollary 3.3]{Li}, and the above form of the result can be found in \cite{nlab_slice}. 
\end{proof}

\begin{prop}[{Slicing Quillen adjunctions, \cite[Proposition 2.5]{Li}, \cite{nlab_slice}}]
	\label{prop:slicing adj}
	Let 
	\[
	L 
	\colon 
	\catC
	\leftrightarrows 
	\catD
	\colon 
	R
	\]
	be a Quillen adjunction and let $A\in \catC$ and $B\in \catD$ be objects. Then the following holds.
	\begin{itemize}
	\item There are Quillen adjunctions 
	\begin{align*}
		L^{A/} 
		\colon 
		\catC^{A/}
		&\leftrightarrows 
		\catD^{L(A)/}
		\colon 
		R^{A/}
		\\
		L^{B/} 
		\colon 
		\catC^{R(B)/}
		&\leftrightarrows 
		\catD^{B/}
		\colon 
		R^{B/}
	\end{align*}
	with $L^{A/}$ (resp. $R^{B/}$) the obvious functors obtained by applying $L$ (resp. $R$) and the adjoint the composition of the application of $R$ (resp. $L$) and the base change along the adjunction (co)unit.
	
	\item Dually, there are analogously defined Quillen adjunctions
	\begin{align*}
		L_{/A} 
		\colon 
		\catC_{/A}
		&\leftrightarrows 
		\catD_{/L(A)}
		\colon 
		R_{/A}
		\\
		L_{/B} 
		\colon 
		\catC_{/R(B)}
		&\leftrightarrows 
		\catD_{/B}
		\colon 
		R_{/B}.
	\end{align*}
\end{itemize}
\end{prop}

\subsection{Unitary colored operads and colored $\La$ operads}
\label{sec:La operads}
A unitary\footnote{This notation is traditional but arguably confusing -- it is not related to the operadic unit.} simplicial colored operad is a colored operad $\POp_*$ such that each space of unary operations is a point,
\[
\POp_*(\underline 0; c) = * \quad \text{for all $c\in \colC$.}	
\]
One may encode such a unitary operad $\POp_*$ equivalently by its suboperad without zero-ary operations 
\[
\POp(\underline r;c)
=
\begin{cases}
	\POp_*(\underline r;c) & \text{if $\underline r\neq \underline 0$} \\
	\emptyset & \text{otherwise}
\end{cases},
\]
together with the set of operations 
\[
	\POp(\underline r;c)
	\to 
	\POp(\underline s;c)
\]
for $\underline s\leq \underline r$, that are obtained by composition with the nullary operations of any color.
We call the latter data a $\La$-structure on the operad $\POp$, and $\POp$ a $\La$-operad.

To formalize the notion, let $\colC\La$ be the category with objects $(\underline r;c)$ and morphisms 
\[
\Mor_{\colC\La}((\underline r;c), (\underline s;d))
=
\begin{cases}
	\{(f_c)_{c\in \colC} \mid 
	f_c : \{1,\dots,r_c\} \hookrightarrow \{1,\dots,s_c\} \}\text{for $c= d$} \\
	\emptyset & \text{for $c\neq d$}
\end{cases}.
\]  
The category of colored $\La$-sequences is then the category of contravariant functors
\[
\La\CsSetSeq := \sSet^{\colC\La^{op}}.
\]
There is a forgetful functor  
\[
\La\CsSetOp \to \La\CsSetSeq
\]
sending a colored $\La$ operad to its underlying $\La$ sequence.
We emphasize that the concepts of $\La$ operad and unitary operad are equivalent, so the reader may opt to replace the former by the latter. However, we shall always assume below that our operads have no zero-ary operations, with those operations encoded by additional algebraic structure in the case of $\La$ operads.

\section{Rational homotopy theory of colored operads and $\La$-operads}
\label{sec:operad rht}

We briefly recall here the construction of a rational homotopy theory of operads and $\La$-operads by the first author from \cite{OperadHomotopyBook, ExtendedRHT}. While these results have been formulated for non-colored operads, the constructions and proofs carry over virtually unchanged to the colored setting, for a finite set of colors $\colC$. 
We shall hence state the results of \cite{OperadHomotopyBook, ExtendedRHT} directly in their generalized colored form, but only briefly remark on the proofs.

\subsection{Model category structures for colored operads and Hopf cooperads}
First, we consider the category $\CsSetOp$ of $\colC$-colored simplicial operads. Recall that we do not allow our operads to have operations in arity zero.
One can construct a cofibrantly generated model structure by right transfer along the adjunction
\begin{align*}
\FreeOp \colon \CsSetSeq &\leftrightarrows \CsSetOp \colon U
\end{align*}
where the right adjoint is the forgetful functor and the left adjoint is the free operad functor. Concretely, this means that the classes of weak equivalences and (co)fibrations in $\CsSetOp$ are as follows.

 \begin{itemize}
	\item The weak equivalences are the morphisms that are weak equivalences of simplicial sets in every arity. 
	\item The fibrations in $\CsSetOp$ are the morphisms that are fibrations of simplicial sets in every arity.
	\item The cofibrations are the morphisms which have the left-lifting property with respect to acyclic fibrations.
	\item The generating (acyclic) cofibrations are the morphisms of the form $\FreeOp(f)$ obtained by applying the free operad functor to the generating (acyclic) cofibrations of the category of simplicial sequences.
\end{itemize}

We refer to \cite[II.8]{OperadHomotopyBook} for the proof of the following statement. 
\begin{prop}
The above distinguished classes of morphisms define a cofibrantly generated model category structure on the category $\CsSetOp$.
\end{prop}

Next, one considers the category of non-negatively graded conilpotent coaugmented colored dg cooperads $\CdgOpc$. 
This category fits into an adjunction 

\begin{align*}
	U \colon \CdgOpc &\leftrightarrows \CdgSeq \colon \FreeOp^c
	,
\end{align*}
with $U$ the forgetful functor and the right adjoint the cofree cooperad functor.
We then define a model category structures on $\CdgOpc$ by left transfer via the above adjunction. Concretely, this means that one defines the distinguished classes of morphisms as follows.
\begin{itemize}
	\item The weak equivalences are the arity-wise quasi-isomorphisms.
	\item The cofibrations in $\CdgOpc$ are the maps that are arity-wise injective in positive cohomological degrees.
	\item The fibrations are all morphisms that have the right-lifting property with respect to the acyclic cofibrations.
	\item The model structure is cofibrantly generated. The generating cofibrations can be taken to be the cofibrations between overall (in all arities together) finite dimensional colored dg cooperads. The generating acyclic cofibrations can be taken to be the cofibrations between arity- and degree-wise bounded cooperads of totally at most countable dimension.
\end{itemize}

\begin{prop}
	The above distinguished classes of morphisms define a cofibrantly generated model category structure on the category $\CdgOpc$.
\end{prop}

For the proof we refer to \cite[Theorem 1.4]{ExtendedRHT}, in which one just needs to replace cooperads by colored cooperads.




Next one considers the category of conilpotent dg Hopf cooperads $\dgHOpc$. One has an adjunction
\begin{align*}
\bbS \colon \CdgOpc &\leftrightarrows \CdgHOpc \colon \omega
,
\end{align*}
with the right-adjoints the forgetful functors, forgetting the commutative algebra structures, and the left-adjoint the arity-wise application of the free commutative algebra functor.
We equip the category $\CdgHOpc$ with a model category structure via right transfer along the above adjunction.
Concretely, this means that:
\begin{itemize}
	\item The weak equivalences in $\CdgHOpc$ are the arity-wise quasi-isomorphisms.
	\item The fibrations in $\CdgHOpc$ are those morphisms that are fibrations in $\CdgOpc$.
	\item The cofibrations are all morphisms that have the left-lifting property with respect to the acyclic fibrations.
	\item The model categories are cofibrantly generated, with the generating (acyclic) cofibrations the images under $\bbS$ of the generating (acyclic) cofibrations in $\CdgOpc$.
\end{itemize}

\begin{prop}
The above distinguished classes of morphisms endow $\CdgHOpc$ with a cofibrantly generated model category structure.
\end{prop}
We refer to \cite{ExtendedRHT,OperadHomotopyBook} (with the replacement of operads by colored operads) for the proof. We will also need:

\begin{prop}[{Colored version of \cite[Proposition 1.7]{ExtendedRHT}, \cite[Lemma II.9.3.10]{OperadHomotopyBook}}]
	\label{prop:cofib aritywise cofib}
Let $f:\POp\to \QOp$ be a cofibration (resp. acyclic cofibration) in $\CdgHOpc$. Then the induced morphisms on each arity component
$$\POp(\underline r;c)\to \QOp(\underline r;c)$$
are cofibrations (resp. acyclic cofibration) of dg commutative algebras.
\end{prop}
\begin{proof}
This holds for the generating (acyclic) cofibrations $I$ (resp. $J$), which are obtained by applying the free commutative algebra functor to cofibrations of dg vector spaces.
Moreover, colimits in $\CdgHOpc$ are created in the category of (colored) dg Hopf collections. Hence all relative $I$-cell complexes (resp- $J$-cell complexes) in $\CdgHOpc$ also have the property of the proposition. But any (acyclic) cofibration is a retract of such a cell complex, see \cite[Proposition II.4.2.1(b)]{OperadHomotopyBook}, and hence is also an arity-wise cofibration of dg commutative algebras.
\end{proof}

\subsection{Rational homotopy theory}\label{sec:col rht op}
Next consider the Quillen adjunction 
\[
\G :  \dgCom \leftrightarrows \sSet^{op}  : \DGOmega
\]
of rational homotopy theory, where 
\begin{equation}\label{equ:G def}
	\G(-)=\Hom_{\sSet}(-, \Omega(\Delta^\bullet))
\end{equation}
and 
\begin{equation}\label{equ:Omega def}
	\DGOmega(-)=\Hom_{\dgCom}(-, \Omega(\Delta^\bullet))
\end{equation}
are Sullivan's piecewise polynomial differential forms, with $\Omega(\Delta^n)=\mathbb Q[t_1,\dots,t_n,dt_1,dt_n]$ the polynomial differential forms on the $n$-simplex.
The functor $\G$ is symmetric monoidal and extends aritywise to a functor from dg Hopf cooperads to simplicial operads. Following \cite{ExtendedRHT} one can define a Quillen right adjoint
\begin{align}\label{equ:G Omega adj op}
\G :  \CdgHOpc &\leftrightarrows \CsSetOp^{op}  : \DGOmega_\sharp
\end{align}
using the adjoint functor theorem, or an explicit construction. One has a comparison morphism $\DGOmega_\sharp(\POp)(\underline r;c) \to \DGOmega(\POp(\underline r; c))$ in each arity, for $\POp$ a colored simplicial operad.
One can furthermore show that $\DGOmega_\sharp$ is an operadic upgrade of the functor $\DGOmega$ in the following sense.

\begin{thm}[colored version of Theorem 2.3 of \cite{ExtendedRHT}] \label{thm:comparison colored}
	Let $\colC$ be a finite poset of colors. 
	Let $\POp$ be a cofibrant object of $\CsSetOp$ 
	such that the following holds:
	\begin{itemize}
		\item For all $c\in C$, $\POp(c;c)$ is connected. 
		\item For all $c,d\in \colC$ we have that $\POp(c;d) =\emptyset$, unless $c\leq d$.
		\item Each $\POp(\underline r;c)$ is of finite rational homology type. 
	\end{itemize}	
	Then the comparison morphisms $\DGOmega_\sharp(\POp)(\underline r;c) \to \DGOmega(\POp(\underline r;c))$ are weak equivalences in each arity.
\end{thm}
\begin{proof}
The proof is identical to that of \cite[Theorem 2.3]{ExtendedRHT}, except that one replaces operads by $\colC$-colored operads. The only caveat is that in the free operad construction appearing in the proof one has to make sure that only finitely many trees contribute to the cohomology in each fixed arity and cohomological degree. This is ensured by our condition that $\colC$ is a finite poset and the vanishing condition on the unary operations. 
\end{proof}

One can apply the functor $\DGOmega_\sharp$ to extend the rational homotopy theory of spaces (simplicial sets) to that of simplicial operads. More concretely, for a cofibrant simplicial colored operad $\QOp$ one defines the rationalization
\[
\QOp^{\QQ} = \DGL\G_\bullet \DGOmega_\sharp(\QOp)
\]
using the left derived functor of the functor $\G$, i.e., $\QOp^{\QQ}\simeq \G(\AOp)$, for $\AOp$ a cofibrant resolution of $\DGOmega_\sharp(\QOp)$. If $\POp$ is another cofibrant simplicial operad it then follows from the adjunction relation that one has a weak equivalence of derived mapping spaces
\[
\Map^h_{\CsSetOp}(\POp,\QOp^{\QQ}) \simeq
\Map^h_{\CdgHOpc}(\DGOmega_\sharp(\QOp),\DGOmega_\sharp(\POp)) \, .
\]

\subsection{Extension to $\La$ operads}
\subsubsection{Model structures on colored $\La$ operads}
The category $\colC\La$ of section \ref{sec:La operads} is a generalized Reedy category.
We may hence endow the category of colored $\La$ sequences $\La\CsSetSeq$ with the Reedy model structure. This can be checked to be the same as the one obtained by left transfer along the forgetful adjunction 
\[
 	\La\CsSetSeq \leftrightarrows \CsSetSeq,
\]
see \cite[II.8.3]{OperadHomotopyBook}.
Concretely, the classes of distinguished morphisms in $\La\CsSetSeq$ are the following:
\begin{itemize}
\item A morphism $f$ is a weak equivalence iff it is an arity-wise weak equivalence on the level of simplicial sets.
\item A morphism is a cofibration iff the underlying map of simplicial colored symmetric sequences is a cofibrations with respect to the projective model structure on $\CsSetSeq$.
\item A morphism is a fibration if it has the right-lifting property with respect to the acyclic cofibrations.
\item The model category structure is cofibrantly generated, see \cite[II.8.3.10]{OperadHomotopyBook}. 
\end{itemize}

Next consider the category of colored $\La$ operads $\La\CsSetOp$. We endow it with a cofibrantly generated model structure by right transfer along the adjunction 
\begin{equation}\label{equ:La op adj}
\FreeOp
\colon 
\La\CsSetSeq
\leftrightarrows
\La\CsSetOp
\colon 
U
\end{equation}
with $U$ the forgetful functor and $\FreeOp$ the free operad functor.
We refer to \cite[II.8.4]{OperadHomotopyBook} for the verification that this defines a valid cofibrantly generated model category structure on $\La\CsSetOp$. 
We call this model structure the Reedy model structure.

Alternatively, one could also consider the projective model structure on $\La\CsSetOp$, 
which is also obtained by transfer along the adjunction \eqref{equ:La op adj}, but where we now equip $\La\CsSetOp$ with the projective model structure.

\subsubsection{Model structure on colored $\La$ Hopf cooperads}
The forgetful functor $U$ from colored dg Hopf $\La$ cooperads to colored dg Hopf cooperads has a left adjoint
\[
(-)\otimes \La \colon 
\CdgHOpc 
\leftrightarrows 
\La\CdgHOpc 
\colon 
U.	
\]
We endow $\La\CdgHOpc$ with a cofibrantly generated model category structure via transfer along this adjunction.
We refer the reader to \cite[II.11.4]{OperadHomotopyBook} for the verification that this construction is valid.

For later reference we need:
\begin{prop}\label{prop:cofib aritywise cofib La}
	Let $f:\POp\to \QOp$ be a cofibration in $\La\CdgHOpc$. Then the induced morphisms on each arity component
	$$\POp(\underline r;c)\to \QOp(\underline r;c)$$
	are cofibrations of dg commutative algebras.
\end{prop}
\begin{proof}
This follows from (the colored analog of) \cite[Theorem II.11.4.12]{OperadHomotopyBook}, which asserts that the forgetful functor $U$ above preserves (acyclic) cofibrations, together with Proposition \ref{prop:cofib aritywise cofib}.
\end{proof}

\subsubsection{Rational homotopy theory of colored $\La$ operads}

Analogously to section \ref{sec:col rht op} above we have an adjunction 
\[
\DGOmega_\sharp \colon 
\La\CsSetOp
\leftrightarrows
(\La\CdgHOpc)^{op}
\colon 
\G,
\]
with $\G$ the arity-wise application of the functor $\G$ of \eqref{equ:G def}.
The left-adjoint agrees with (lifts) the functor $\DGOmega_\sharp$ on non-$\La$ operads of \eqref{equ:G Omega adj op}, see \cite[Theorem 4.3]{ExtendedRHT}.
The adjunction above is furthermore a Quillen adjunction if we equip $\La\CsSetOp$ with the Reedy model structure, see \cite[Proposition 4.4]{ExtendedRHT}.
We also note that this property holds as well (a fortiori) for the projective model structure on $\La\CsSetOp$.

Furthermore, since the functor $\DGOmega_\sharp$ defined for $\La$ operads agrees with that for plain operads, we have that the comparison morphisms
\[
	\DGOmega_\sharp(\POp)(\underline r;c)
	\to 
	\DGOmega(\POp(\underline r;c))
\]
are weak equivalences, provided the conditions of Theorem \ref{thm:comparison colored} are satisfied.

\subsection{Simplicial enrichment}
\label{sec:enrichment}
\begin{prop}[Fresse]
The categories of colored simplicial operads $\CsSetOp$ and of colored simplicial $\La$ operads $\La\CsSetOp$ are simplicial model categories. 
\end{prop}
\begin{proof}
The non-colored version of the statement can be found in \cite[section II.8.5.6]{OperadHomotopyBook}, the verification for the colored case is identical.
We just briefly recall here that the function object $\POp^K$ associated to a simplicial set $K$ and a colored operad $\POp$ is defined such that 
\[
	\POp^K(\underline r;c) = \POp(\underline r;c)^K
\]
agrees with the corresponding function object on simplicial sets. 
The pullback-corner axiom then follows from the corresponding statement for simplicial sets.
The tensoring $\POp\otimes K$ is defined by adjunction.
\end{proof}

We cannot show that the category of colored dg Hopf cooperads $\CdgHOpc$ is a simplicial model category. However, following \cite{FWSimplicial} it at least carries a weaker structure, that suffices to state an explicit model for the mapping spaces. We define these models for the mapping spaces (simplical sets) such that 
\[
	\Map(\COp, \DOp)_n := 
	\Mor_{\CdgHOpc_{\Omega(\Delta^n)}}(\COp\otimes \Omega(\Delta^n), \DOp\otimes \Omega(\Delta^n))
\]
where $\CdgHOpc_{\Omega(\Delta^n)}$ denotes the category of colored dg Hopf cooperads over the dg commutative algebra $\Omega(\Delta^n)$ and 
\begin{align*}
	(\COp\otimes \Omega(\Delta^n)) (\underline r; c) := \COp(\underline r; c)\otimes \Omega(\Delta^n).
\end{align*}

The face and degeneracy maps on $\Map(\COp, \DOp)_\bullet$ 
are induced by the face and degeneracy maps of the simplicial dg commutative algebra $\Omega(\Delta^\bullet)$ 
in the evident manner. 
The mapping spaces thus defined inherit natural composition morphisms, and we have $\Map(\COp, \DOp)_0=\Mor_{\CdgHOpc}(\COp, \DOp)$.
The mapping spaces defined above hence define a simplicial enrichment of the category $\CdgHOpc$.
Furthermore, one has an adjunction relation, see \cite[Proposition 2]{FWSimplicial},
\begin{equation}\label{equ:map adjunction}
\Mor_{\sSet}(K, \Map(\COp, \DOp))
\cong 
\Mor_{\CdgHOpc}(\COp, \DOp^K),
\end{equation}
with $\DOp^K$ a dg Hopf cooperad naturally associated to the dg Hopf cooperad $\DOp$ and the simplicial set $K$. To define $\DOp^K$ one uses that any dg Hopf cooperad $\DOp$ can be written as a reflexive equalizer of cofree objects
\[
\DOp \cong \mathrm{eq}\left(
\begin{tikzcd}
\FreeOp^c(\DOp) \ar[shift left]{r}\ar[shift right]{r}& 
\FreeOp^c(\FreeOp^c(\DOp))  \ar[bend right]{l}
\end{tikzcd}
\right).
\]
One then defines 
\[
\DOp^K \cong \mathrm{eq}\left(
\begin{tikzcd}
\FreeOp^c(\DOp\otimes \Omega(K)) \ar[shift left]{r}\ar[shift right]{r}& 
\FreeOp^c(\FreeOp^c(\DOp)\otimes \Omega(K))  \ar[bend right]{l}
\end{tikzcd}
\right).
\]
This construction satisfies the pullback-corner axiom (see \cite[Proposition 3]{FWSimplicial}): For a cofibration of simplicial sets $i:K\to L$ and a fibration of dg Hopf cooperads $p:\COp\to \DOp$ the induced morphism 
\[
\COp^L \to \DOp^L \times_{\DOp^K} \COp^K
\]
is a fibration in $\CdgHOpc$, and a weak equivalence if either $i$ or $p$ is a weak equivalence.
The pullback-corner axiom is the main axiom of a simplicial model category. However, the simplicial model category axioms would also require that the functor $\DOp\mapsto \DOp^K$ is a right adjoint. This is not true in our case since it does not preserve arbitrary limits, although it preserves finite limits, cf. \cite[Section 1]{FWSimplicial}.

The analogous construction also defines simplicial mapping spaces on the category $\La\CdgHOpc$. 

Generally, recall that in any model category $\catC$ one can define mapping spaces (simplicial sets) between objects $\AOp$ and $\BOp$ by taking simplicial or cosimplicial framings. Concretely, for $\BOp_\bullet$ a simplicial framing of $\BOp$ the model categorical mapping space may be defined as the simplicial set
\[
\Mor_{\catC}(\AOp, \BOp_\bullet).
\]
We refer to \cite[Chapter II.3]{OperadHomotopyBook} for a recollection of the general theory.
Now, if the category $\catC$ has a simplicial cotensor structure $(-)^K$, which satisfies the pullback-corner axiom, then $\BOp^{\Delta^\bullet}$ is a simplicial framing for the object $\BOp$.
Hence in our situation one directly obtains:

\begin{cor}[{Colored verion of \cite[Theorem 4 and Proposition 31]{FWSimplicial}}]
The definitions above yield models for (i.e., are weakly equivalent to) the mapping spaces in the model categories $\CdgHOpc$ and $\La\CdgHOpc$. 
\end{cor}
\begin{proof}
The proof is identical to that of \cite[Theorem 4 and Proposition 31]{FWSimplicial}, with the substitution of colored (co)operads for (co)operads.
\end{proof}


\section{Rational homotopy theory of triples}
\label{sec:triples}

\subsection{Model category structures}
We next consider triples $(\POp, \MOp, \QOp)$ consisting of two simplicial operads $\POp$ and $\QOp$ and a $\POp$-$\QOp$-operadic bimodule $\MOp$.
We denote the category of such objects $\sSetTrip$. 

We have a free/forgetful adjunction
\begin{align}\label{equ:tripadjunctions}
	\FreeOp \colon \sSetSeq^3 &\leftrightarrows \sSetTrip \colon U.
\end{align}
Here the forgetful functor $U$ associates to a triple $(\POp, \MOp, \QOp)$ the corresponding triple $(\POp, \MOp, \QOp)$ of the underlying symmetric sequences in simplicial sets.
The free functor associates to a triple $(\AOp, \BOp, \COp)$ of symmetric sequences the triple 
\begin{align*}
	\FreeOp(\AOp, \BOp, \COp) &= (\FreeOp(\AOp), \FreeOp_{\FreeOp(\AOp),\FreeOp(\COp)}\BOp, \FreeOp(\COp))
\end{align*}
consisting of two free operads and a free bimodule of these operads.

We equip the category $\sSetTrip$ with a cofibrantly generated model category structure via transfer along the adjunction \eqref{equ:tripadjunctions}. Concretely, this means that:
\begin{itemize}
	\item The weak equivalences in $\sSetTrip$ are the morphisms that are arity-wise weak equivalences of simplicial sets.
	\item The fibrations in $\sSetTrip$ are those morphisms that are arity-wise fibrations of simplicial sets. 
	\item The cofibrations are the morphisms that have the left-lifting property with respect to the acyclic fibrations.
	\item The model category structure is cofibrantly generated, with the generating (acyclic) cofibrations the images under $\FreeOp$ of the generating (acyclic) cofibrations in $\sSetSeq^3$.
\end{itemize}

\begin{prop}\label{prop:model cat trip}
The above distinguished morphisms equip the category $\sSetTrip$ with a cofibrantly generated model category structure.
\end{prop}
\begin{proof}
	This is a special case of \cite[Theorem 2.1]{BMColored}, since the above triples are algebras over a suitable colored operad.
\end{proof}

\subsection{$\sSetTrip$ as coreflective subcategory}

The category $\sSetTrip$ can be realized as a coreflective subcategory of the category of two-colored operads 
\begin{align}\label{equ:Trip Col adj}
	\iota \colon \sSetTrip &\leftrightarrows \CsSetOp \colon \pi
	,
\end{align}
where $\colC=\{I,II\}$ is the two-element set.
The right-adjoint functor is defined such that 
\[
	\pi(\AOp) = (\POp, \MOp, \QOp)
\]
with 
\begin{align*}
	\POp(r) &= \AOp(r,0; I) 
	&
	\QOp(r) &= \AOp(0,r; II) 
	&
	\MOp(r) &= \AOp(0,r; I),
\end{align*}
equipped with the operadic and bimodule structures obtained by restriction of the operadic composition in $\AOp$. 

The left-adjoint functor $\iota$ is the two colored operad
\[
	\iota(\POp, \MOp, \QOp)
	= 
	\FreeOp(\POp\sqcup \MOp\sqcup \QOp)/\sim,
\]
with $\POp$ considered in arity $(-,0,I)$, $\QOp$ considered in arity $(0,-,II)$ and $\MOp$ considered in arity $(0,-,I)$, and with the relations derived from those on $\POp$, $\QOp$ and $\MOp$.
In particular, we have that 
\[
\iota(\POp, \MOp, \QOp)(r_I, r_{II};c)
=
\begin{cases}
	\POp(r_I) & \text{for $c=I$, $r_{II}=0$} \\
	\QOp(r_{II}) & \text{for $c=II$, $r_{I}=0$} \\
	\MOp(r_{II}) & \text{for $c=I$, $r_{I}=0$}
\end{cases},
\]
so that it is clear that $F\circ \iota =\mathit{id}$.
The inclusion $\iota$ is furthermore fully faithful, so that the
adjunction \eqref{equ:Trip Col adj} realizes the category $\sSetTrip$ as a coreflective subcategory of $\CsSetOp$.

\begin{lemm}
The adjunction \eqref{equ:Trip Col adj} is Quillen.
Furthermore, the inclusion $\iota$ is homotopically fully faithful, i.e., for any $X,Y\in \sSetTrip$ we have 
\[
\Map^h_{\sSetTrip}(X, Y)
\simeq 
\Map^h_{\CsSetOp}(L\iota(X), L\iota(Y)).
\]
\end{lemm}
\begin{proof}
The fibrations (resp. weak equivalences) in both $\CsSetOp$ and $\sSetTrip$ are those morphisms that are arity-wise fibrations (resp. weak equivalences) of simplicial sets.
Hence it is obvious that the right-adjoint $\pi$ preserves those classes, and is hence right Quillen.

For the final statement on homotopical full faithfulness we invoke Lemma \ref{lem:fully faithful}, using that $\pi$ preserves weak equivalences.

\end{proof}

\subsection{Dg Hopf triples and model category structure}
We next consider the category $\dgHTripc$ of triples $(\COp, \NOp, \DOp)$ consisting of dg Hopf cooperads $\COp$ and $\DOp$ and a Hopf $\COp$-$\DOp$ bicomodule $\NOp$.
As in the previous section, we have an adjunction 
\begin{align}\label{equ:Trip Col adj}
	\pi \colon \CdgHOpc &\leftrightarrows \dgHTripc \colon \iota
	,
\end{align}
with $\colC=\{I,II\}$. We have $\pi \circ \iota=\mathit{id}$, the functor $\iota$ is fully faithful, and the above adjunction realizes $\dgHTripc$ as a reflective subcategory of $\CdgHOpc$.
We define a cofibrantly generated model category structure on $\dgHTripc$ by right transfer along $\iota$.


\begin{prop}\label{prop:Htrip model}
Right transfer along $\iota$ endows the category $\dgHTripc$ 
with a well-defined cofibrantly generated model category structure. The weak equivalences of this model structure are the quasi-isomorphisms.
\end{prop}

We will need:

\begin{lemm}\label{lem:pre Htrip model}
A morphism $f$ of $\dgHTripc$ is a quasi-isomorphism iff $\iota(f)$ is a quasi-isomorphism.
\end{lemm}
\begin{proof}
	Let $f$ be a morphism of $\dgHTripc$.
	If $\iota(f)$ is a quasi-isomorphism in $\CdgHOpc$, then so is $f=\pi\iota(f)$, since $\pi$ preserves quasi-isomorphisms. Conversely, suppose that $f$ is a quasi-isomorphism.
	Note that the arity components of $\iota f$ are built from $f$ via taking tensor products, direct sums or coinvariants under finite group actions. Since these operations preserve quasi-isomorphisms, $\iota f$ is a quasi-isomorphism.
\end{proof}

\begin{proof}[Proof of Proposition \ref{prop:Htrip model}]
We use Theorem \ref{thm:transfer Kan} to check well-definedness of the model structure.
The domains of the images under $\pi$ of the generating (acyclic) cofibration of $\CdgHOpc$ are overall at most countably dimensional, and hence countably small.
Thus the first condition of Theorem \ref{thm:transfer Kan} holds.

For the second condition, we use that by Proposition \ref{prop:cofib aritywise cofib} the acyclic cofibrations $J$ of $\CdgHOpc$ are arity-wise acyclic cofibrations of dg commutative algebras.
Hence the same holds for the generating cofibrations $\pi(J)$ of $\dgHTripc$.
But since colimits in $\dgHTripc$ are created arity-wise in dg commutative algebras, it immediately follows that all $\pi(J)$-relative cell complexes are aritywise weak equivalences of dg commutative algebras. But this is the same as quasi-isomorphisms. Hence the image of any such morphism under $\iota$ is again a quasi-isomorphism by Lemma \ref{lem:pre Htrip model}, and hence a weak equivalence in $\CdgHOpc$.
This shows the second condition of Theorem \ref{thm:transfer Kan} and hence establishes the model structure on $\dgHTripc$.

The final assertion on weak equivalences follows immediately from Lemma \ref{lem:pre Htrip model}.



\end{proof}

As a direct consequence of Lemma \ref{lem:fully faithful} we then obtain:

\begin{cor}
The inclusion $\iota:\dgHTripc\to \CdgHOpc$ is homotopically fully faithful. 
\end{cor}

\subsection{Rational homotopy theory of triples}

\begin{prop}\label{prop:rht triples}
	The functor $\G$ on $\dgHTripc$ obtained by arity-wise application of \eqref{equ:G def} has a left Quillen adjoint
	\begin{equation}\label{equ:rht trip adj}
		\DGOmega_\sharp \colon \sSetTrip \leftrightarrows
		(\dgHTripc)^{op} \colon \G.
    \end{equation} 
	The left-adjoint $\DGOmega_\sharp$ is obtained by restriction/projection with the functors $\iota, \pi$ of the functor $\DGOmega_\sharp$ defined on colored operads in \eqref{equ:RHT adjunction col}.
\end{prop}
\begin{proof}
The existence of the adjoint functor follows from the adjoint functor lifting theorem \cite[Theorem II.4.5.6]{Borceux}, applied to the square 
\[
    \begin{tikzcd}
        \sSetTrip 
        \ar[shift left, dashed]{r}{\DGOmega_\sharp}
        \ar[shift left]{d}{U}
        &
        (\dgHTripc)^{op}
        \ar[shift left]{l}{\G}
        \ar[shift left]{d}{U}
        \\
        \sSetSeq^3
        \ar[shift left]{r}{\DGOmega}
        \ar[shift left]{u}{\FreeOp}
        &
        ((\dgca\Seq)^3)^{op}
        \ar[shift left]{l}{\G}
        \ar[shift left]{u}{\FreeOp^c}
    \end{tikzcd},
\]
Note that the vertical downwards (forgetful) functors are monadic, and the lower horizontal arrow $\G$ has a left adjoint, hence so does the upper horizontal arrow by the adjoint functor lifting theorem.

Our functor then fits into a square of adjunctions
\begin{equation}\label{equ:omega square}
\begin{tikzcd}
	\sSetTrip 
	\ar[shift left, dashed]{r}{\Omega_\sharp^{trip}}
	\ar[shift left]{d}{\iota}
	&
	(\dgHTripc)^{op}
	\ar[shift left]{l}{\G}
	\ar[shift left]{d}{\iota}
	\\
	\CsSetOp
	\ar[shift left]{r}{\Omega_\sharp}
	\ar[shift left]{u}{\pi}
	&
	(\CdgHOpc)^{op}
	\ar[shift left]{l}{\G}
	\ar[shift left]{u}{\pi}
\end{tikzcd},
\end{equation}
where we have temporarily renamed our new left-adjoint functor to $\DGOmega_\sharp^{trip}$, to distinguish it from the functor $\DGOmega_\sharp$ for colored operads in the bottom row.
It is obvious that the diagram of right adjoints commutes,
\[
\G \pi = \pi \G,
\]
since $\G$ is an arity-wise application of a functor and $\pi$ just projects to suitable arity-components.
Hence it follows that the diagram of left-adjoints also commutes, i.e.,
\[
	\iota \Omega_\sharp^{trip} = \Omega_\sharp \iota.
\]
Applying $\pi$ from the left we obtain 
\[
	\Omega_\sharp^{trip} = \pi \Omega_\sharp \iota
\]
as desired.

\end{proof}

\subsection{The case of $\La$ triples}
\subsubsection{Reedy model structure on  $\La\sSetTrip$}
Let $\La\sSetTrip$ be the category of triples $(\POp, \MOp, \QOp)$ consisting of two simplicial $\La$-operads $\POp$, $\QOp$ and an operadic $\POp$-$\QOp$ $\La$-bimodule.
We equip this category $\La\sSetTrip$ with a cofibrantly generated model category structure by right transfer along the free/forgetful adjunction 
\[
\FreeOp \colon (\La\sSetSeq)^3 \leftrightarrows \La\sSetTrip \colon 
U.
\]
Here we consider the left-hand model category with the Reedy model structure, and call the resulting model structure on $\La\sSetTrip$ the Reedy model structure as well. 

\begin{prop}
Right transfer along the above adjunction equips $\La\sSetTrip$ with a well-defined cofibrantly generated model category structure, that we call the Reedy model structure.
\end{prop}
\begin{proof}
We desire to apply Theorem \ref{thm:transfer Kan}.

Condition (1) of that theorem holds true, since the domains of the generating are built from simplicial sets with an at most countable number of simplices. Hence these objects are countably small.

For condition (2) of Theorem \ref{thm:transfer Kan} consider the diagram of adjunctions 
\begin{equation*}
	\begin{tikzcd}
		(\La\sSetSeq)^3
		\ar[shift left]{r}{\FreeOp}
		\ar[shift left]{d}{\iota}
		&
		\La\sSetTrip
		\ar[shift left]{l}{U}
		\ar[shift left]{d}{\iota}
		\\
		\La\CsSetSeq
		\ar[shift left]{r}{\FreeOp}
		\ar[shift left]{u}{\pi}
		&
		\La\CsSetOp
		\ar[shift left]{l}{U}
		\ar[shift left]{u}{\pi}
	\end{tikzcd}.
	\end{equation*}
	The left and bottom adjunctions are Quillen. The diagram of right adjoints commutes, since obviously $\pi U = U\pi$, and hence also the diagram of left-adjoints. Furthermore, the left-hand morphism $\pi$ and the bottom morphism $U$ preserve weak equivalences.
	Hence we may apply Lemma \ref{lem:cell we criterion} to the upper-right triangle of our diagram above, to see that condition (2) of Theorem \ref{thm:transfer Kan} holds.

\end{proof}

\subsubsection{Model structure on $\La\dgHTripc$} 
We define a cofibrantly generated model category structure on $\La\dgHTripc$ by right transfer along the adjunction 
\[
	\pi \colon 
	\La\CdgHOpc \leftrightarrows \La\dgHTripc
	\colon 
	\iota.
\]
that lifts the adjunction \eqref{equ:Trip Col adj} to $\La$ dg Hopf triples.
\begin{prop}\label{prop:La dgHTripc model str}
This yields a well-defined cofibrantly generated model category structure on $\La\dgHTripc$. The weak equivalences in that model structure are the quasi-isomorphisms.
\end{prop}
\begin{proof}
We invoke again Theorem \ref{thm:transfer Kan}, and check its conditions.
Condition (1) holds true since the generating (acyclic) cofibrations have domains of overall at most countable dimension, and are hence countably small.

Next, as in Lemma \ref{lem:pre Htrip model} it follows that the weak equivalences of the transfered model structure are the quasi-isomorphisms.

For condition (2) of Theorem \ref{thm:transfer Kan} we use that the generating acyclic cofibrations are the images under $\pi$ of generating acyclic cofibrations of $\La\CdgHOpc$. Hence they are in particular acyclic cofibrations of dg commutative algebras arity-wise by Proposition \ref{prop:cofib aritywise cofib La}. But colimits in $\La\dgHTripc$ are created arity-wise on the level of dg commutative algebras, and hence our relative cell complexes are again arity-wise quasi-isomorphisms, and hence weak equivalences as desired.
\end{proof}

\subsubsection{Rational homotopy theory of $\La$ triples}

\begin{prop}
We have a Quillen adjunction 
\[
\DGOmega_\sharp \colon 
\La\sSetTrip
\leftrightarrows
(\La\dgHTripc)^{op}
\colon 
\G,
\]
with $\G$ the arity-wise application of the functor $\G$ of \eqref{equ:G def}.
The left-adjoint lifts the functor $\DGOmega_\sharp$ of Proposition \ref{prop:rht triples}.
\end{prop}
\begin{proof}
This follows analogously to Proposition \ref{prop:rht triples}. 
\end{proof}

\subsection{Simplicial enrichment}
\label{sec:enrichment trip}
Finally, we remark that the simplicial enrichment of the category $\CdgHOpc$ of colored dg Hopf cooperads recalled in section \ref{sec:enrichment} can be extended to the category $\dgHTripc$ of dg Hopf triples.
Analogously to section \ref{sec:enrichment} we define the mapping spaces between triples $\SOp, \TOp\in \dgHTripc$ as the simplicial sets 
\begin{equation}\label{equ:trip map def}
\Map(\SOp, \TOp):=
\Mor_{\dgHTripc_{\Omega(\Delta^\bullet)}}
(\SOp\otimes \Omega(\Delta^\bullet), \TOp\otimes \Omega(\Delta^\bullet))
\end{equation}
by just extending our ground ring to $\Omega(\Delta^\bullet)$.
We next claim that these simplicial mapping spaces fit into an adjunction relation 
\begin{equation}\label{equ:adj map triples}
    \Mor_{\sSet}(K, \Map(\SOp, \TOp))
    \cong 
    \Mor_{\dgHTripc}(\SOp, \TOp^K),
\end{equation}
analogous to \eqref{equ:map adjunction}.
Here the construction $\TOp^K$ associated to a simplicial set $K$ and a dg Hopf triple $\TOp=(\COp, \MOp,\DOp)$ is defined as follows.
For a cofree object $\TOp=\FreeOp^c(\AOp)$ we set $\TOp^K=\FreeOp^c(\AOp\otimes\Omega(K))$.
Then an arbitrary object $\TOp$ can be written as a reflexive equalizer of free objects
\[
\TOp\cong \eq\left(
    \begin{tikzcd}
    \FreeOp^c(\TOp^0) \ar[shift left]{r}\ar[shift right]{r} &
     \FreeOp^c(\TOp^1) \ar[bend right]{l}
    \end{tikzcd}
\right)
\]
and we extend the construction $(-)^K$ so that it preserves reflexive equalizers, i.e., 
\[
\TOp:= \eq\left(
    \begin{tikzcd}
    \FreeOp^c(\TOp^0\otimes \Omega(K)) \ar[shift left]{r}\ar[shift right]{r} &
     \FreeOp^c(\TOp^1\otimes \Omega(K)) \ar[bend right]{l}
    \end{tikzcd}.
\right)
\]
To see that \eqref{equ:adj map triples} holds we just note that both sides agree for $\TOp$ a cofree object, and both sides preserve equalizers in $\TOp$.

\begin{prop}
	\label{prop:Trip Op PCA}
    \begin{enumerate}
        \item
The inclusion $\iota:\dgHTripc\to \CdgHOpc$ is compatible with the simplicial cotensor structures $(-)^K$ on both sides. That is, for any simplicial set $K$ and dg Hopf triple $\TOp\in \dgHTripc$ we have that 
\[
\iota(\TOp^K) \cong (\iota(\TOp))^K
\]
\item (Pullback-corner axiom) For a cofibration of simplicial sets $i:K\to L$ and a fibration of dg Hopf triples $p:\SOp\to \TOp$ in $\dgHTripc$ the induced morphism 
\begin{equation}\label{equ:pcp trip}
\SOp^L \to \TOp^L \times_{\TOp^K} \TOp^K
\end{equation}
is a fibration in $\dgHTripc$, and a weak equivalence if either $i$ or $p$ is a weak equivalence.
    \end{enumerate}
\end{prop}
\begin{proof}
For the first assertion note that the inclusion functor $\iota$ sends cofree triples to cofree dg Hopf cooperads.
On such cofree objects we verify 
\[
\iota(\FreeOp^c(\AOp)^K)
=
\iota(\FreeOp^c(\AOp\otimes \Omega(K)))
=
\FreeOp^c(\AOp\otimes \Omega(K))
=
(\iota(\FreeOp^c(\AOp)))^K.
\]
Then, for general $\TOp$ the assertion then follows since $\iota$ preserves equalizers by adjunction.

For the second assertion of the proposition note that the model category structure on $\dgHTripc$ is defined by transfer along $\iota$. Hence to check that \eqref{equ:pcp trip} is an (acyclic) fibration we need to check that 
\[
   \iota( \SOp^L) \to \iota(\TOp^L \times_{\TOp^K} \TOp^K)
\]
is an (acyclic) fibration in $\CdgHOpc$. But this is true by the first assertion of the proposition and the pullback-corner axiom in $\CdgHOpc$, see section \ref{sec:enrichment}.
\end{proof}

The analogous construction and verification also remains true for the catgeory $\La \dgHTripc$ of $\La$ dg Hopf triples. 
The pullback-corner axiom implies that for any object $\TOp\in \dgHTripc$ (resp. in $\La \dgHTripc$) we have that $\TOp^{\Delta^\bullet}$ is a simplicial framing for $\TOp$. Hence we obtain that our "ad hoc" definition of mapping spaces is valid:

\begin{cor}
    For $\SOp,\TOp\in \dgHTripc$ (resp. in $\La \dgHTripc$) the object \eqref{equ:trip map def} is weakly equivalent to the mapping space between $\SOp$ and $\TOp$ defined canonically via simplicial framings in the model category $\dgHTripc$ (resp. in $\La \dgHTripc$).
\end{cor}

\section{Rational homotopy theory of operadic bimodules}
\label{sec:bimodules}

\subsection{Model category structure}
Next, we fix two simplicial operads $\POp$, $\QOp$, and we consider the category $\BiMod_{\POp,\QOp}$ of $\POp$-$\QOp$ operadic bimodules.
This category comes with a free/forgetful adjunction
\begin{align}
\FreeOp_{\POp,\QOp} \colon 
\sSetSeq
&\leftrightarrows 
\BiMod_{\POp,\QOp} \colon U
.
\end{align}
We equip the category $\BiMod_{\POp,\QOp}$ with a cofibrantly generated model structure by right transfer along this adjunction:
\begin{itemize}
	\item The weak equivalences in $\BiMod_{\POp,\QOp}$ are the morphisms that are arity-wise weak equivalences of simplicial sets.
	\item The fibrations in $\BiMod_{\POp,\QOp}$ are those morphisms that are fibrations arity-wise on the level of simplicial sets. 
	\item The cofibrations are all morphisms that have the left-lifting property with respect to the acyclic fibrations.
	\item The model category structure is cofibrantly generated, with the generating (acyclic) cofibrations the images under the free bimodule functor $\FreeOp_{\POp,\QOp}$ of the generating (acyclic) cofibrations in $\sSetSeq$.
\end{itemize}

\begin{prop}\label{prop:model cat bimod}
The above classes of distinguished morphisms define a cofibrantly generated model structure on the category $\BiMod_{\POp,\QOp}$.
\end{prop}
\begin{proof}
	This is again a special case of \cite[Theorem 2.1]{BMColored}, since $\POp$-$\QOp$-bimodules are algebras over a suitable colored operad.
\end{proof}

\subsection{$\BiMod_{\POp,\QOp}$ as coreflective subcategory}
We consider the under-categories $\sSetTrip^{(\POp, \emptyset, \QOp)/}$ and $\CsSetOp^{\iota(\POp, \emptyset, \QOp)/}$ associated to the triple $(\POp, \emptyset, \QOp)\in \sSetTrip$. (Here $\iota$ is the inclusion \eqref{equ:Trip Col adj}.)
The category $\BiMod_{\POp,\QOp}$ can then be realized as a coreflective subcategory in both of these undercategories via the adjunctions
\begin{equation}\label{equ:trip under adj}
	\begin{tikzcd}
	\BiMod_{\POp,\QOp} \ar[shift left]{r}{\iota_B}
	&
	\ar[shift left]{l}{\pi_B}
	\sSetTrip^{(\POp, \emptyset, \QOp)/}
	\ar[shift left]{r}{\iota}
	&
	\ar[shift left]{l}{\pi}
	\CsSetOp^{\iota(\POp, \emptyset, \QOp)/}
	\end{tikzcd}.
\end{equation}
Here the left-adjoint $\iota_B$ sends a bimodule $\MOp$ to the canonical morphism of triples $(\POp, \emptyset, \QOp)\to (\POp,\MOp,\QOp)$. The right adjoint $\pi_B$ is defined on a morphism 
$f: (\POp, \emptyset, \QOp)\to (\POp',\MOp,\QOp')$ as $F_B(f) = f^* \MOp$, i.e., we pull back (restrict) the given $\POp'$-$\QOp'$ bimodule structure on $\MOp$ to a $\POp$-$\QOp$ bimodule structure via $f$. The right-hand adjunction is produced from the adjunction \eqref{equ:Trip Col adj} via Proposition \ref{prop:slicing adj}. Just note that there is no base-change in the definition of the right-adjoint since $\pi\iota\cong \mathit{id}$, so we just apply the functors $\pi$, $\iota$, justifying our (abuse of) notation.

We equip the undercategories above with the usual slice model structure, see Theorem \ref{thm:slice model str}.
Then the right-hand adjunction in \eqref{equ:trip under adj} is automatically a Quillen adjunction by Proposition \ref{prop:slicing adj}.

\begin{lemm}
	The left-hand adjunction $(\iota_B,\pi_B)$ in \eqref{equ:trip under adj} is also Quillen. The left-adjoint $\iota_B$ is homotopically fully faithful.
\end{lemm}
\begin{proof}
The right adjoint $\pi_B$ acts on the level of the underlying symmetric sequences by projecting to the middle component $\MOp$. But since weak equivalences and fibrations are created on the level of simplicial symmetric sequences the result immediately follows.

Fully faithfulness follows from Lemma \ref{lem:fully faithful}.
\end{proof}

\subsection{Hopf bicomodules}
Let $\COp$ and $\DOp$ be Hopf cooperads.
We consider the category $\dgHBiModc_{\COp,\DOp}$ of dg Hopf $\COp$-$\DOp$-bicomodules.
This category fits into a chain of reflective subcategories 
\begin{equation}\label{equ:Htrip over adj}
	\begin{tikzcd}
		(\CdgHOpc)_{/\iota(\COp,*,\DOp)}
		\ar[shift left]{r}{\pi}
	&
	\ar[shift left]{l}{\iota}
	(\dgHTripc)_{/(\COp,*,\DOp)}
	\ar[shift left]{r}{\pi_B}
	&
	\ar[shift left]{l}{\iota_B}
	\dgHBiModc_{\COp,\DOp}
	\end{tikzcd}.
\end{equation}

Here the right-adjoint $\iota_B$ sends a Hopf $\COp$-$\DOp$-bicomodule $\MOp$ to the canonical morphism of triples $(\COp, \MOp, \DOp)\to (\COp, *, \DOp)$.
The left-adjoint $\pi_B$ sends a morphism of triples $f:(\COp', \MOp, \DOp')\to (\COp, *, \DOp)$ to the corestriction $f_*\MOp$.

We define a cofibrantly generated model category structure on the category $\dgHBiModc_{\COp,\DOp}$ by transfer from the slice model category structure on $(\dgHTripc)_{/(\COp,*,\DOp)}$.
Concretely, this means that:
\begin{itemize}
\item The weak equivalences of $\dgHBiModc_{\COp,\DOp}$ are the quasi-isomorphisms.
\item The fibrations are those morphisms $f$ for which $\iota(f)$ is a fibration in $(\dgHTripc)_{/(\COp,*,\DOp)}$.
\item The cofibrations are the morphisms that have the left-lifting property with respect to the acyclic fibrations.
\end{itemize}

\begin{prop}\label{prop:HBiModc model str}
The above distinguished classes of morphisms define a cofibrantly generated model category structure on $\dgHBiModc_{\COp,\DOp}$.
Moreover, the inclusions $\iota_B$ and $\iota \circ \iota_B$ are homotopically fully faithful.
\end{prop}
\begin{proof}
We follow the proof of Proposition \ref{prop:Htrip model}. We apply Theorem \ref{thm:transfer Kan} to transfer the model structure along the adjunction $(\pi_B, \iota_B)$ above. Condition (1) of Theorem \ref{thm:transfer Kan} is again satisfied because the codomains of the generating (acyclic) cofibrations are of at most countable dimension, and hence countably small. 

Also note that the weak equivalences in the transferred model structure on $\dgHBiModc_{\COp,\DOp}$ are the quasi-isomorphisms, since the same is true in $(\dgHTripc)_{/(\COp,*,\DOp)}$, and $f$ is a quasi-isomorphism iff $\iota_B(f)$ is.

For condition (2) of Theorem \ref{thm:transfer Kan} we use that colimits in $\dgHBiModc_{\COp,\DOp}$ are created arity-wise on the level of dg commutative algebras. Furthermore, the generating acyclic cofibrations are arity-wise acyclic cofibrations of dg commutative algebras by Proposition \ref{prop:cofib aritywise cofib}.
Hence the relative cell complexes with respect to the generating acyclic cofibrations are aritywise relative cell complexes with respect to acyclic cofibrations in dg commutative algebras, and hence quasi-isomorphisms. This verifies the second condition of Theorem \ref{thm:transfer Kan}. We hence obtain a well-defined cofibrantly generated model category structure on 
$\dgHBiModc_{\COp,\DOp}$ by transfer, and 
the adjunction $(\pi_B, \iota_B)$ is Quillen.

Finally, the homotopical fully faithfulness of $\iota_B$ follows again from Lemma \ref{lem:fully faithful}.
\end{proof}

\subsection{Rational homotopy theory}

\begin{prop}\label{prop:bimod rht adj}
Let $\POp$ and $\QOp$ be simplicial operads and set $\COp:=\DGOmega_\sharp(\POp)$ and $\DOp:=\DGOmega_\sharp(\QOp)$.
Then there is a Quillen adjunction 
\begin{equation}\label{equ:bimod rht adj}
\DGOmega_\sharp \colon \BiMod_{\POp, \QOp} \leftrightarrows \left(\dgHBiModc_{\COp,\DOp}\right)^{op} \colon \G,
\end{equation}
where the right-adjoint $\G$ is the arity-wise application of the functor $\G$ of \eqref{equ:G def}.
The left-adjoint $\DGOmega_\sharp$ is defined from the corresponding functor $\DGOmega_\sharp$ on triples of Proposition \ref{prop:rht triples} such that 
\[
\DGOmega_\sharp(\POp, \MOp, \QOp) = 
(\DGOmega_\sharp(\POp), \DGOmega_\sharp(\MOp), \DGOmega_\sharp(\QOp)).
\]
\end{prop}
\begin{proof}
Consider the following diagram 
\[
    \begin{tikzcd}
    \BiMod_{\POp, \QOp} 
    \ar[shift left]{r}{\DGOmega_\sharp} 
    \ar[shift left]{d}{\iota_{B}}
    & \left(\dgHBiModc_{\COp,\DOp}\right)^{op} 
    \ar[shift left]{d}{\iota_{B}}
    \\
    \sSetTrip^{(\POp, \emptyset,\QOp)/}
    \ar[shift left]{r}{\DGOmega_\sharp} 
    \ar[shift left]{u}{\pi_B}
    & 
    \left( (\dgHTripc)_{/(\COp,*,\DOp)} \right)^{op}
    \ar[shift left]{u}{\pi_B}
    \ar[shift left]{l}{\G}
    \end{tikzcd}.
    \]
Here we define the upper horizontal functor as stated in the proposition, as the "middle part" of the functor $\DGOmega_\sharp$ on triples.
The Quillen adjunction in the bottom row is obtained by slicing the Quillen adjunction \eqref{equ:rht trip adj}, see also Proposition \ref{prop:slicing adj}.
By construction we have that $\iota_B\DGOmega_\sharp=\iota_B\DGOmega$.
Hence we can apply Lemma \ref{lem:restr adj} to obtain a right adjoint $\G'$ of the upper horizontal functor $\Omega_\sharp$, that satisfies $\pi_B\G=\G'\pi_B$. Applying $\iota_B$ from the right we obtain $\G'=\pi_B\G\iota_B$, so that $\G'$ is indeed the arity-wise application of the functor $\G$ of \eqref{equ:G def}. We will hence rename $\G=\G'$ and drop the notation $\G'$.

Furthermore, from the second assertion of Lemma \ref{lem:restr adj} we immediately conclude that \eqref{equ:bimod rht adj} is a Quillen adjunction.

\end{proof}

From Theorem \ref{thm:comparison colored} we then immediately conclude:

\begin{cor}
Let $\POp$ and $\QOp$ be simplicial operads and let $\MOp$ be an operadic $\POp$-$\QOp$-bimodule. Assume that these data satisfy the following conditions:
\begin{itemize}
\item $\POp(1)$ and $\QOp(1)$ are connected.
\item $\POp(r)$, $\QOp(r)$ and $\MOp(r)$ have degree-wise finite dimensional rational homology for each $r$.
\end{itemize}
Then the left-adjoint functor $\DGOmega_\sharp$ of Proposition \ref{prop:bimod rht adj} is such that the canonical comparison morphism 
\[
\DGOmega_\sharp(\MOp)(r) \to \DGOmega(\MOp(r))	
\]
is a weak equivalence of dg commutative algebras (i.e., a quasi-isomorphism) for each $r$.
\end{cor}

\subsection{$\La$ bimodules}
\subsubsection{Model structure on simplicial $\La$ bimodules}
Let $\POp$, $\QOp$ be fixed simplicial $\La$-operads, and consider the category of $\POp$-$\QOp$-$\La$-bimodules $\La\BiMod_{\POp, \QOp}$. 

We define a cofibrantly generated model category structure on $\La\BiMod_{\POp, \QOp}$ by right transfer along the free/forgetful adjunction
\[
	\FreeOp_{\POp, \QOp} \colon \La\sSetSeq \leftrightarrows \Lai\BiMod_{\POp, \QOp}
	\colon U.
\]

\begin{prop}
Right transfer along the above adjunction endows the category $\La\BiMod_{\POp, \QOp}$ with a well-defined cofibrantly generated model category structure.
\end{prop}
\begin{proof}
We apply again Theorem \ref{thm:transfer Kan}.
Condition (1) of the Theorem is satisfied since the domains of the generating (acyclic) cofibrations are free $\POp$-$\QOp$-bimodules, and the generating simplicial symmetric sequence have overall only countable many simplices.
The domains are hence countably small objects.
 
For condition (2) we consider the diagram of adjunctions
\begin{equation*}
	\begin{tikzcd}
		\La\sSetSeq
		\ar[shift left]{r}{\FreeOp_{\POp, \QOp}}
		\ar[shift left]{d}{\iota}
		&
		\Lai\BiMod_{\POp, \QOp}
		\ar[shift left]{l}{U}
		\ar[shift left]{d}{\iota_B}
		\\
		\left((\La\sSetSeq)^3\right)^{(\POp, \emptyset, \QOp)/}
		\ar[shift left]{r}{\FreeOp}
		\ar[shift left]{u}{\pi}
		&
		\La\sSetTrip^{(\POp, \emptyset, \QOp)/}
		\ar[shift left]{l}{U}
		\ar[shift left]{u}{\pi}
	\end{tikzcd}.
	\end{equation*}

The diagram of right adjoints and the diagram of left adjoints commute.
Furthermore, the left-hand morphism $\pi$ and the bottom morphism $U$ preserve weak equivalences.
	Hence we may apply Lemma \ref{lem:cell we criterion} to the upper-right triangle of our diagram above, to see that condition (2) of Theorem \ref{thm:transfer Kan} holds.
\end{proof}

\subsubsection{Model category structure on dg Hopf $\La$ bicomodules} 
Let $\COp$, $\DOp$ be dg $\La$ Hopf cooperads and consider the category $\La\dgHBiModc_{\COp, \DOp}$ of dg $\La$ Hopf $\COp$-$\DOp$-bicomodules.
We define a cofibrantly generated model structure on the the category $\La\dgHBiModc_{\COp, \DOp}$ via right transfer along the adjunction 
\[
	\pi \colon \Lai\dgHTripc_{/(\COp, *, \DOp)} \leftrightarrows \Lai\dgHBiModc_{\COp, \DOp} \colon \iota
\]

\begin{prop}
This yields a well-defined cofibrantly generated model category structure on $\La\dgHBiModc_{\COp, \DOp}$.
\end{prop}
\begin{proof}
	This is analogous to the proof of Proposition \ref{prop:La dgHTripc model str}.
\end{proof}

\subsubsection{Rational homotopy theory of $\La$ bimodules}

\begin{prop}\label{prop:bimod rht adj La}
	Let $\POp$ and $\QOp$ be simplicial $\La$ operads and set $\COp:=\Omega_\sharp(\POp)$ and $\DOp:=\Omega_\sharp(\QOp)$.
	Then there is a Quillen adjunction 
	\[
	\Omega_\sharp \colon \Lai\BiMod_{\POp, \QOp} \leftrightarrows \left(\Lai\dgHBiModc_{\COp,\DOp}\right)^{op} \colon \G,
	\]
	where the right-adjoint $\G$ is the arity-wise application of the functor $\G$ of \eqref{equ:G def}.
	The left-adjoint $\Omega_\sharp$ is defined from the corresponding functor $\Omega_\sharp$ on triples of Proposition \ref{prop:rht triples} such that 
	\[
	\Omega_\sharp(\POp, \MOp, \QOp) = 
	(\Omega_\sharp(\POp), \Omega_\sharp(\MOp), \Omega_\sharp(\QOp)).
	\]
	\end{prop}
	\begin{proof}
This is analogous to the proof of Proposition \ref{prop:bimod rht adj La}. 
	\end{proof}

\subsection{Simplicial enrichment}

\subsubsection{... on the over-category $\dgHTripc_{/(\COp,*\DOp)}$}

The simplicial enrichment on the model category $\dgHTripc$ gives rise to a simplicial enrichment structure on the over-category $\dgHTripc_{/X}$, with $X=(\COp,*\DOp)$ (or any other object). Unfortunately, we are not aware of significant literature on the topic of enriched (co)slice model categories, although the theory seems to be known to experts, see \cite{nlab_slice_enriched,nlab_slice}. So we recall here some constructions.

The mapping spaces $\Map_{/X}(f,g)$ in $\dgHTripc_{/X}$ between objects $f:\SOp\to X$ and $g:\TOp\to X$ fit into a pullback square
\[
\begin{tikzcd}
\Map_{/X}(f,g) \ar{r} \ar{d} & \Map(\SOp, \TOp) \ar{d}{\circ g} \\
* \ar{r}{f} &  \Map(\SOp, X)
\end{tikzcd}.
\]
Similarly, there is a simplicial cotensor structure $(f,K)$ associating to the object $f:\TOp\to X$ of $\dgHTripc_{/X}$ and the simplicial set $K$ the object $\tilde f^K:\tilde \TOp^K\to X$ of $\dgHTripc_{/X}$ fitting into a pullback square
\[
\begin{tikzcd}
    \tilde \TOp^K \ar{r} \ar{d}{\tilde f^K} & \TOp^K \ar{d}{f^K}\\
X=X^* \ar{r} &  X^K
\end{tikzcd}.
\]
Here we use the notation $\tilde f^K$ to distinguish our new morphism from the morphism $f^K:\TOp^K\to X^K$ obtained by applying the simplicial cotensor structure of $\dgHTripc$ to $f$.
We then have the adjunction relation 
\begin{equation}\label{equ:adj map trip over}
    \Mor_{\sSet}(K, \Map_{/X}(f, g))
    \cong 
    \Mor_{\dgHTripc_{/X}}(f, \tilde g^K),
\end{equation}
that follows from \eqref{equ:adj map triples} and the fact that limits can be taken out of the morphism sets in the second slot.

Finally, the simplicial cotensoring on $\dgHTripc_{/X}$ also satisfies the pullback-corner property, see \cite[Proposition 2.3]{nlab_slice}.

Let us make the cotensor structure more explicit for the specific case $X=(\COp,*\DOp)$ relevant for this paper.
Consider an object of the form $f:(\COp,\MOp,\DOp) \to (\COp,*\DOp)$ of the over-category. Then the object $\tilde f^K$ is defined by the pullback square 
\begin{equation}\label{equ:overc cotensor pb}
\begin{tikzcd}
   (\COp, \tilde \MOp^K, \DOp) \ar{r} \ar{d}{\tilde f^K} & (\COp^K, \MOp^K, \DOp^K) \ar{d}{f^K}\\
   (\COp, *, \DOp) \ar{r} &  (\COp^K, *, \DOp^K)
\end{tikzcd}.
\end{equation}
Here the $\COp$-$\DOp$-bicomodule $\tilde \MOp^K$ is obtained by coinduction from the $\COp^K$-$\DOp^K$-bicomodule $\MOp^K$.
In particular, if $\MOp=\FreeOp^c_{\COp,\DOp}(\AOp)$ is a cofree bicomodule, then $\MOp^K=\FreeOp^c_{\COp^K,\DOp^K}(\AOp\otimes \Omega(K))$ and $\tilde \MOp^K=\FreeOp^c_{\COp,\DOp}(\AOp\otimes \Omega(K))$.

\subsection{... on $\dgHBiModc_{\COp, \DOp}$}
Finally, we construct a simplicial enrichment on the category $\dgHBiModc_{\COp, \DOp}$, following sections \ref{sec:enrichment} and \ref{sec:enrichment trip} above. 
We define mapping spaces (simplicial sets) between objects $\MOp, \NOp\in \dgHBiModc_{\COp, \DOp}$ as
\begin{equation}\label{equ:bimod map def}
    \Map(\MOp, \NOp):=
    \Mor_{\dgHBiModc_{\COp\otimes \Omega(\Delta^\bullet), \DOp\otimes \Omega(\Delta^\bullet)} }
    (\MOp\otimes \Omega(\Delta^\bullet), \NOp\otimes \Omega(\Delta^\bullet)),
\end{equation}
by changing the base ring to $\Omega(\Delta^\bullet)$.
We claim that these mapping spaces fit into an adjunction 
\begin{equation}\label{equ:adj map bimod}
    \Mor_{\sSet}(K, \Map(\MOp, \NOp))
    \cong 
    \Mor_{\dgHBiModc_{\COp, \DOp}}(\MOp, \hat \NOp^K),
\end{equation}
analogous to \eqref{equ:map adjunction}, \eqref{equ:adj map triples}. The simplicial cotensoring $(\NOp, K)\mapsto\hat \NOp^K$ can be constructed as in earlier sections. However, we might as well restrict it from the discussion of the previous subsection.
Consider the following diagram
\[
    \begin{tikzcd}[row sep=1.7cm, column sep=2cm]
        \sSet 
        \ar[shift left]{rd}{(\widetilde{\iota_B(\NOp)})^{(-)}}
        \ar{r}{\tilde \NOp^{(-)}} 
    & \left(\dgHBiModc_{\COp,\DOp} \right)^{op}
    \ar[shift left]{d}{\iota_{B}}
    \\
    & 
    \left(\dgHTripc_{/(\COp,*,\DOp)} \right)^{op}
    \ar[shift left]{u}{\pi_B}
    \ar[shift left]{lu}{\Map_{/(\COp,*,\DOp)}(-,\iota_B(\NOp))}
    \end{tikzcd}.
\]
From the discussion of the previous subsection we know that the diagonal arrows are a Quillen adjunction. The horizontal arrow uses the object in the upper left-hand corner of \eqref{equ:overc cotensor pb}. In other words, we have that $\iota_B(\tilde N^K) = (\widetilde{\iota_B(\NOp)})^{K}$, so the triangle of left adjoints in the above diagram commutes. We may hence apply Lemma \ref{lem:restr adj} again, considering our triangle as a quenched square. The Lemma asserts that $\tilde \NOp^{(-)}$ has the right adjoint 
\[
\Map_{/(\COp,*,\DOp)}(\iota_B(-), \iota_B(N))
=
\Map(-, N),
\] 
using fully faithfullness of $\iota_B$ for the equality. Hence we have established the adjunction \eqref{equ:adj map bimod}, with 
\[
    \hat \NOp^K := \tilde \NOp^K.
\]



By the same argument as in the proof of Proposition \ref{prop:Trip Op PCA} above we then conclude:

\begin{prop}
(Pullback-corner axiom) For a cofibration of simplicial sets $i:K\to L$ and a fibration of dg Hopf bicomodules $p:\MOp\to \NOp$ in $\dgHBiModc_{\COp, \DOp}$ the induced morphism 
\begin{equation}\label{equ:pcp trip}
\hat \MOp^L \to \hat \NOp^L \times_{\hat \NOp^K} \hat \MOp^K
\end{equation}
is a fibration in $\dgHBiModc_{\COp, \DOp}$, and a weak equivalence if either $i$ or $p$ is a weak equivalence.
\end{prop}

The analogous construction and verification also remains true for the catgeory $\La \dgHBiModc_{\COp, \DOp}$ of $\La$ dg Hopf bicomodules over the $\La$ dg Hopf cooperads $\COp$ and $\DOp$. 
As before the pullback-corner axiom implies that for any bicomodule $\MOp$ we have that $\hat \MOp^{\Delta^\bullet}$ is a simplicial framing for $\MOp$, so that we obtain:

\begin{cor}
    For $\MOp,\NOp\in \dgHBiModc_{\COp, \DOp}$ (resp. in $\La  \dgHBiModc_{\COp, \DOp}$) the object \eqref{equ:bimod map def} (resp. its $\La$ variant) is weakly equivalent to the mapping space between $\MOp$ and $\NOp$ defined canonically via simplicial framings in the model category $\dgHBiModc_{\COp, \DOp}$ (resp. in $\La \dgHBiModc_{\COp, \DOp}$).
\end{cor}

\section{Generalizations and special cases}
\label{sec:mixedmod}
\subsection{Extension to bimodules over colored operads}
Above we considered $\POp$-$\QOp$-bimodules for ordinary (single colored) operads $\POp$ and $\QOp$. The constructions and proofs readily extend to the case of bimodules over colored operads.
To this, we fix two disjoint finite posets of colors $\colC_1$ and $\colC_2$.
We also consider the disjoint union 
\[
\colC := \colC_1 \sqcup \colC_2.
\]
We extend the poset structures on $\colC_1$ and $\colC_2$ to a poset structure  on $\colC$ by declaring that $c_1>c_2$ for all $c_1\in \colC_1$, $c_2\in \colC_2$.

The constructions of section \ref{sec:triples} generalize to the colored setting. In particular we may define a model category $\sSet\Trip_{\colC_1,\colC_2}$ of triples $(\POp,\MOp,\QOp)$ consisting of a $\colC_1$-colored operad $\POp$, a $\colC_2$-colored operad $\QOp$ and a $\POp$-$\QOp$ operadic bimodule.
For a fixed simplicial $\colC_1$-colored operad $\POp$ and a $\colC_2$-colored operad $\QOp$ we may furthermore define the category of $\POp$-$\QOp$ operadic bimodules $\BiMod_{\POp,\QOp}$.
Then Propositions \ref{prop:model cat trip} and \ref{prop:model cat bimod} above extend to the following result, which is proven identically:
\begin{prop}
Let $\colC_1$ and $\colC_2$ be finite sets. 
Let $\POp$ be a $\colC_1$-colored operad and let $\QOp$ be a $\colC_2$-colored operad in simplicial sets.
Then the categories $\sSet\Trip_{\colC_1,\colC_2}$ and $\BiMod_{\POp,\QOp}$ carry model category structures such that the weak equivalences (resp. fibrations) are those morphisms that are arity-wise weak equivalences (resp. fibrations) of simplicial sets.
Furthermore, one has a chain of Quillen adjunctions 
\[
	\begin{tikzcd}
        \BiMod_{\POp,\QOp} \ar[shift left]{r}{\iota_B}
        &
        \ar[shift left]{l}{\pi_B}
        \sSetTrip_{\colC_1,\colC_2}^{(\POp, \emptyset, \QOp)/}
        \ar[shift left]{r}{\iota}
        &
        \ar[shift left]{l}{\pi}
        \CsSetOp^{\iota(\POp, \emptyset, \QOp)/}
     \end{tikzcd},
\]
with $\colC:=\colC_1\sqcup \colC_2$ and $\iota$, $\iota_B$, $\pi$, $\pi_B$ the natural inclusion and projection functors.
\end{prop}

Dually we consider the category $\dgHTripc_{\colC_1,\colC_2}$ 
of triples $(\COp,\MOp,\DOp)$ with $\COp$ a $\colC_1$-colored dg Hopf cooperad, $\DOp$ a $\colC_2$-colored dg Hopf cooperad and $\MOp$ a $\COp$-$\DOp$-bicomodule. For fixed such $\COp$, $\DOp$ we also consider the category $\dgHBiModc_{\COp,\DOp}$ of $\COp$-$\DOp$-bicomodules.
These categories fit into adjunctions 
\begin{align}\label{equ:Trip Col adj col}
	\pi \colon \CdgHOpc &\leftrightarrows \dgHTripc_{\colC_1,\colC_2} \colon \iota
	,
\end{align}
and 
\begin{equation}\label{equ:Htrip over adj col}
	\begin{tikzcd}
		(\CdgHOpc)_{/\iota(\COp,*,\DOp)}
		\ar[shift left]{r}{\pi}
	&
	\ar[shift left]{l}{\iota}
	(\dgHTripc_{\colC_1,\colC_2})_{/(\COp,*,\DOp)}
	\ar[shift left]{r}{\pi_B}
	&
	\ar[shift left]{l}{\iota_B}
	\dgHBiModc_{\COp,\DOp}
	\end{tikzcd}.
\end{equation}
Parallel to Propositions \ref{prop:Htrip model} and \ref{prop:HBiModc model str} above one then shows:
\begin{prop}
    Let $\colC_1$ and $\colC_2$ be finite posets.
Then the category $\dgHTripc_{\colC_1,\colC_2}$ may be equipped with a well defined model category structure via right transfer along the adjunction \eqref{equ:Trip Col adj col}. 

Let $\COp$ be a $\colC_1$-colored dg Hopf cooperad and $\DOp$ a $\colC_2$-colored dg Hopf cooperad. Then the category $\dgHBiModc_{\COp,\DOp}$ may be equipped with a well-defined model category structure via right transfer along the right-hand adjunction of \eqref{equ:Htrip over adj col}.
\end{prop}

Finally, we obtain a rational homotopy adjunction between the above categories by slicing and restricting the adjunction \eqref{equ:G Omega adj op}, as in the proof of Propositions \ref{prop:rht triples} and \ref{prop:bimod rht adj} above. This yields:
\begin{prop}\label{prop:colored mod rht adj}
    Let $\colC_1$ and $\colC_2$  be finite posets.
    \begin{itemize}
    \item There is a Quillen adjunction 
    \[
        \DGOmega_\sharp \colon \sSetTrip_{\colC_1, \colC_2} \leftrightarrows \left(\dgHTripc_{\colC_1, \colC_2}\right)^{op} \colon \G,
    \]
    where the right-adjoint $\G$ is the arity-wise application of the functor $\G$ of \eqref{equ:G def}.
    \item Let $\POp\in \colC_1\sSetOp$ and $\QOp\in \colC_2\sSetOp$ be colored simplicial operads. Set $\COp:=\DGOmega_\sharp(\POp)$ and $\DOp:=\DGOmega_\sharp(\QOp)$.
    Then there is a Quillen adjunction 
    \[
        \DGOmega_\sharp \colon \BiMod_{\POp, \QOp} \leftrightarrows \left(\dgHBiModc_{\COp,\DOp}\right)^{op} \colon \G,
    \]
    where the right-adjoint $\G$ is the arity-wise application of the functor $\G$ of \eqref{equ:G def}.
    \end{itemize}
\end{prop}

\subsection{Infinitesimal bimodules}

We specialize the constructions for colored $\POp$-$\QOp$-bimodules of the previous subsection to the following situation:
\begin{itemize}
    \item The set of colors are $\colC_1=\{I,II\}$ and $\colC_2=\{*\}$.
    \item The $\colC_1$-colored operad $\POp=\AssM$ is the operad governing pairs $(A,M)$ of an associative algebra $A$ and a left $A$-module $M$.
    Explicitly, this means that 
    \begin{align*}
    \AssM(p,q;c) =
    \begin{cases}
    \bbS_p & \text{for $q=0$, $c=I$} \\
    \bbS_p & \text{for $q=1$, $c=II$} \\
    \emptyset & \text{otherwise}
    \end{cases}
    \end{align*}
    \item The (single-colored) simplicial operad $\QOp$ is arbitrary.
\end{itemize}
In this case $\POp$-$\QOp$-bimodules are the same as pairs $(\AOp,\MOp)$ consisting of an algebra object $\AOp$ in the category of right $\QOp$-modules, and a right module $\MOp$ of $\AOp$.

The canonical example of an algebra object in right $\QOp$-modules is $\AOp = \QOp_1$, defined such that 
\[
\QOp_1(r) = \QOp(r+1),
\]
with the composition at the first input providing the algebra structure.
Following Arone and Turchin, the module $\MOp$ in this situation is then called an infinitesimal $\QOp$-bimodule, see \cite{AroneTurchin}.

For $\QOp$ a simplicial operad we denote the category of infinitesimal $\QOp$-bimodules by $\IBiMod_{\QOp}$.
There is then an adjunction 
\begin{equation}\label{equ:IBiMod adj}
\iota \colon \IBiMod_{\QOp} \rightleftarrows \BiMod_{\AssM, \QOp}^{(\QOp_1,\emptyset)/} \colon \pi,
\end{equation}
such that $\iota$ is the natural inclusion sending an infinitesimal $\QOp$ bimodule $\MOp$ to the $\AssM$-$\QOp$-bimodule $(\QOp_1,\MOp)$.
The right adjoint $\pi$ sends a morphism $(f_1,f_2): (\QOp_1,\emptyset)\to (\AOp,\MOp)$ to the infinitesimal $\QOp$-bimodule $f_1^*\MOp$ obtained by restriction along $f_1$.
\begin{prop}
There is a model category structure on $\IBiMod_{\QOp}$ such that the weak equivalences (res. fibrations) are those morphisms that are arity-wise weak equvalences (res. fibrations) of simplicial sets.
With this model category structure the adjunction \eqref{equ:IBiMod adj} is Quillen.
\end{prop}  
\begin{proof}
The model structure is obtained by right transfer along the forgetful functor 
\[
    \IBiMod_{\QOp} \to \sSetSeq.
\]
The well-definedness of this model structure follows again from \cite[Theorem 2.1]{BMColored}.
It is clear that the adjunction is Quillen since the right-adjoint clearly preserves weak equivalences and fibrations. 
(Note that on the level of simplicial sets $f_1^*\MOp(r) = \MOp(r)$.)
\end{proof}

Dually, we consider for a dg Hopf cooperad $\DOp$ the category $\dgHIBiModc_{\DOp}$ of infinitesimal $\DOp$ bicomodules.
Analogusly to \eqref{equ:IBiMod adj} It fits into an adjunction 
\begin{equation}\label{equ:IBiModc adj}
    \pi \colon (\dgHBiModc_{\AssMc, \DOp})_{/(\DOp_1,*)}  \rightleftarrows \dgHIBiModc_{\DOp} \colon \iota.
\end{equation}

\begin{prop}
    Let $\DOp$ be a dg Hopf cooperad. Then the category $\dgHIBiModc_{\DOp}$ of infinitesimal $\DOp$ bicomodules may be equipped with a model category structure by right transfer along the adjunction \eqref{equ:IBiModc adj}.
    The weak equivalences of this model category structure are the arity-wise quasi-isomorphisms.
\end{prop}
\begin{proof}
This is parallel to the proof of Proposition \ref{prop:HBiModc model str}.
\end{proof}

\begin{prop}\label{prop:ibimod rht adj}
    Let $\POp$ be a simplicial operad and let $\COp:=\Omega_\sharp(\POp)$.
    Then there is a Quillen adjunction 
    \[
    \DGOmega_\sharp \colon \IBiMod_{\POp} \leftrightarrows \left(\dgHIBiModc_{\COp}\right)^{op} \colon \G,
    \]
    where the right-adjoint $\G$ is the arity-wise application of the functor $\G$ of \eqref{equ:G def}.
\end{prop}
\begin{proof}
By slicing the Quillen adjunction of the second assertion of Proposition \ref{prop:colored mod rht adj} we obtain a Quillen adjuntion 
\[
    \DGOmega_\sharp\colon
    \BiMod_{\AssM, \POp}^{(\POp_1, \emptyset)/}
 \leftrightarrows
\left( (\dgHBiModc)_{/(\COp_1,*)} \right)^{op}
    \colon\G.
\]
We then apply Lemma \ref{lem:restr adj} to restrict this adjunction to the (co)reflective subcategories $\IBiMod_{\POp}$, respectively $\dgHIBiModc_{\COp}$, as in the proof of Proposition \ref{prop:bimod rht adj}.


\end{proof}

\appendix 

\section{Induction and restriction for bimodules}

\subsection{Induction and restriction}
Let $f_1: \POp\to \POp'$ and $f_2:\QOp\to \QOp'$ be morphisms of simplicial operads, and let $f: (\POp,\emptyset, \QOp)\to (\POp',\emptyset, \QOp')$ be the corresponding morphism in $\sSetTrip$.
Then by general facts about slice categories (see Proposition \ref{prop:slice base change} above) we have a Quillen adjunction between the under-categories 
\begin{equation}\label{equ:slice res ind 1}
	f_* \colon \sSetTrip^{(\POp,\emptyset, \QOp)/}
	 \leftrightarrows 
	 \sSetTrip^{(\POp',\emptyset, \QOp')/}
	 \colon f^*.
\end{equation}
Here the right-adjoint is the pre-composition with the morphism $f$. The left-adjoint takes $\alpha:(\POp,\emptyset, \QOp)\to (\AOp, \MOp, \BOp)$ to the pushout
\[
\begin{tikzcd}
	(\POp,\emptyset, \QOp)\ar{r}{\alpha} \ar{d}{f}&  (\AOp, \MOp, \BOp) \ar[dashed]{d}\\
	(\POp',\emptyset, \QOp') \ar[dashed]{r}{f_*\alpha} & X.
\end{tikzcd}.
\]

In the particular case that $\AOp=\POp$ and $\BOp=\QOp$ we have that 
\[
X = (\POp', f_*\MOp, \QOp'),	
\]
with $f_*\MOp$ the induced $\POp'$-$\QOp'$-bimodule.
Furthermore, by Proposition \ref{prop:slice base change} we have that \eqref{equ:slice res ind 1} is a Quillen equivalence if either $f$ is an acyclic cofibration or a weak equivalence between cofibrant objects. This in turn is equivalent to either $f_1$ and $f_2$ being acyclic cofibrations or 
weak equivalences between cofibrant operads.


Our adjunction may be further restricted to bimodules.
In this setting the sufficient assumptions for Quillen equivalence may be relaxed, as shown independently by Fresse and Berger-Moerdijk

\begin{thm}[Fresse \cite{OperadModulesFunctors}, Berger-Moerdijk \cite{BMColored}]
	\label{thm:ind res bimod}
Let $f_1:\POp\to \POp'$ and $f_2:\QOp\to\QOp'$ be morphisms of simplicial operads, and let $f=(f_1,f_2)$.
Then there is a Quillen adjunction 
\begin{equation}\label{equ:bimod res ind}
	f_* \colon \BiMod_{\POp, \QOp}
	 \leftrightarrows 
	 \BiMod_{\POp, \QOp}
	 \colon f^*
\end{equation}
with the right-adjoint being restriction along $f$ and the left-adjoint induction. If $f_1$ and $f_2$ are weak equivalences and $\POp$, $\POp'$ are $\Sigma$-cofibrant then \eqref{equ:bimod res ind} is a Quillen equivalence.
\end{thm}

\subsection{Corestriction and coinduction}
Let $f_1:\COp\to \COp'$ and $\f_2:\DOp\to \DOp'$ be morphisms of dg Hopf cooperads.
Then, dualizing the discussion in the last subsection, we obtain the following result from Proposition \ref{prop:slice base change}.
\begin{prop}
Let $f_1:\COp\to \COp'$, and $f_2:\DOp\to\DOp'$ be morphisms of dg Hopf cooperads, and let $f=(f_1,f_2)$. Then there is a Quillen adjunction of the over-categories 
\begin{equation}\label{equ:slice cores coind 1}
	f_* \colon \dgHTripc_{/(\COp,*, \DOp)}
	 \leftrightarrows 
	 \dgHTripc_{/(\COp',*, \DOp')}
	 \colon f^*.
\end{equation}
If $f_1$, $f_2$ are weak equivalences between fibrant objects, or acyclic fibrations, then \eqref{equ:slice cores coind 1} is a Quillen equivalence.
\end{prop}

Again we may restrict further to the category of bicomodules.
Similarly to Theorem \ref{thm:ind res bimod} we may relax the assumptions necessary for a Quillen equivalence.

\begin{thm}
    Let $f_1:\COp\to \COp'$, and $f_2:\DOp\to\DOp'$ be morphisms of dg Hopf cooperads, and let $f=(f_1,f_2)$. Then the coinduction and corestriction functors induce a Quillen adjunction
    \begin{equation}\label{equ:slice cores coind 1}
        f_* \colon \dgHBiModc_{\COp, \DOp}
         \leftrightarrows 
         \dgHBiModc_{\COp',\DOp'}
         \colon f^*.
    \end{equation}
    If $f_1$ and $f_2$ are weak equivalences, then \eqref{equ:slice cores coind 1} is a Quillen equivalence.

    Furthermore, if $\COp, \COp', \DOp, \DOp'$ are dg $\La$ Hopf cooperads and the morphisms $f_1,f_2$ preserve the $\La$ structures, then the coinduction and corestriction functors form a Quillen adjunction 
    \begin{equation}\label{equ:slice cores coind 1 la}
        f_* \colon \La\dgHBiModc_{\COp, \DOp}
         \leftrightarrows 
         \La\dgHBiModc_{\COp',\DOp'}
         \colon f^*,
    \end{equation}
    which is a Quillen equivalence if $f_1$ and $f_2$ are weak equivalences.

\end{thm}
\begin{proof}

We first show that \eqref{equ:slice cores coind 1} is a Quillen adjunction. There are multiple ways. But the one that is in the style of the arguments in this paper is to consider the following diagram,
\[
    \begin{tikzcd}
        \dgHBiModc_{\COp, \DOp} 
    \ar[shift left]{d}{\iota_{B}}
    & \dgHBiModc_{\COp',\DOp'} 
    \ar[shift left]{l}{f^*} 
    \ar[shift left]{d}{\iota_{B}}
    \\
    \dgHTripc_{/(\COp,\emptyset, \DOp)}
    \ar[shift left]{r}{f_*} 
    \ar[shift left]{u}{\pi_B}
    & 
    \dgHTripc_{/(\COp',\emptyset, \DOp')}
    \ar[shift left]{u}{\pi_B}
    \ar[shift left]{l}{f^*}
    \end{tikzcd},
\]
with the upper horizontal $f^*$ the coinduction along $f$. By construction we have $\iota_B f^*=f^*\iota_B$. Applying (the dual of) Lemma \ref{lem:restr adj} we conclude that $\pi_Bf_*\iota_B$ is the left adjoint to $f^*$. This functor is the corestriction functor, and we denote it by $f_*$ as well. Furthermore, the second assertion of Lemma \ref{lem:restr adj} states that eqref{equ:slice cores coind 1} is a Quillen adjunction. 

Now suppose that $f$ is a weak equivalence, i.e., $f_1$ and $f_2$ are quasi-isomorphisms of dg Hopf cooperads.
Note that the left adjoint $f_*$ (corestriction) is the identity on the level of symmetric sequences of dg commutative algebras and hence $f_*$ creates weak equivalences.
To show that \eqref{equ:slice cores coind 1} is a Quillen equivalence, it hence sufffices to check that for every fibrant $\COp'$-$\DOp'$-cobimodule $\MOp$ the adjunction counit 
\[
f_* f^*\MOp \to \MOp  
\]
is a weak equivalence.
Next, note that \eqref{equ:slice cores coind 1} fits into a commutative diagram 
\[
\begin{tikzcd}
	\dgHBiModc_{\COp, \DOp} \ar[shift left]{r}{f_*}\ar{d} & 
	\dgHBiModc_{\COp', \DOp'} \ar[shift left]{l}{f^*}\ar{d}
	\\
	\dgBiModc_{\COp, \DOp} \ar[shift left]{r}{f_*}& 
	\dgBiModc_{\COp', \DOp'}\ar[shift left]{l}{f^*}
\end{tikzcd},
\]
with the vertical arrows the natural forgetful functors from Hopf bicomodules to plain dg bicomodules.
In particular, note that the adjoint functors in the top and bottom row are identical, since limits in Hopf bicomodules are created in plain dg bicomodules.
Furthermore, the vertical forgetful functors preserve (in fact create) fibrations, and our $\MOp$ above is also fibrant in $\dgBiModc_{\COp', \DOp'}$.
Hence it suffices to check that for any fibrant object $\NOp$ of $\dgBiModc_{\COp', \DOp'}$ we have that the adjunction counit
\begin{equation}\label{equ:N counit}
	f_* f^*\NOp \to \NOp  
\end{equation}
is a weak equivalence. In fact, we do not need to show this for all such $\NOp$, but only $\NOp=\Bar\Bar^c X$ that are bar-cobar-constructions, since any object of $\dgBiModc_{\COp', \DOp'}$ is weakly equivalent to one such.
But for bar cobar constructions, the morphism \eqref{equ:N counit} may be explicitly written down and reads 
\[
	f_* f^*\Bar\Bar^c X
	=
	\COp \circ (\Bar^c\COp') \circ X \circ (\Bar^c\DOp')\circ \DOp 
	\xrightarrow{f_1\circ \mathit{id} \circ  \mathit{id} \circ  \mathit{id} \circ f_2}
	\COp' \circ (\Bar^c\COp') \circ X \circ (\Bar^c\DOp')\circ \DOp'= \Bar\Bar^c X.
\]
The differential is composed of the internal differentials on the 4 cooperads and $X$, and pieces using the cooperadic cocompositions and the coaction.
The total cohomological degree of some element is the sum of the cohomological degrees of the leading and trailing factors $\COp$ and $\DOp$ (resp. $\COp'$ and $\DOp'$) and of the middle piece  
\[
\Bar^c X:= (\Bar^c\COp') \circ X \circ (\Bar^c\DOp').
\]
Both degrees are non-negative (using that we work with non-negatively graded cochain complexes throughout), and we obtain a first-quadrant spectral sequence from the filtration by the total cohomological degree on $\Bar^c X$.
The associated graded morphism to our counit is, displaying the differentials explicitly
\[
	(\COp,d_{\COp}) \circ (\Bar^c X,0)\circ (\DOp,d_{\DOp}) 
	\xrightarrow{f_1\circ \mathit{id} \circ f_2}
	(\COp',d_{\COp'}) \circ (\Bar^c X,0)\circ (\DOp',d_{\DOp'}).
\] 
By the K\"unneth Theorem and the assumption that $f_1$, $f_2$ are weak equivalences this is a weak equivalence.  
Hence by the spectral sequence comparison theorem our counit morphism is a quasi-isomorphism as well as desired.

The proof for $\La$ cooperads is identical. Note that still the limits are generated in the underlying categories of plain dg cobimodules, and the bar-cobar resolution has a natural $\La$-structure.
\end{proof}

\bibliographystyle{plain}
\bibliography{MappingSpaceModel}

\end{document}